\documentclass[11pt,a4paper]{article}
\usepackage[T2A,T1]{fontenc}
\usepackage[utf8]{inputenc}
\usepackage{lmodern}
\DeclareFontFamily{T2A}{lmr}{}
\DeclareFontShape{T2A}{lmr}{m}{n}{<->ssub * cmr/m/n}{}
\usepackage[russian,english]{babel}
\usepackage[margin=25mm]{geometry}
\usepackage{amsmath,amssymb,amsthm}
\usepackage{graphicx,array,longtable}
\usepackage{algorithm,algpseudocode}
\usepackage{framed,xcolor,url}
\usepackage[hidelinks]{hyperref}
\allowdisplaybreaks
\newcommand{\R}{\mathbb{R}}
\newcommand{\Z}{\mathcal{Z}}
\newcommand{\X}{\mathcal{X}}
\newcommand{\Y}{\mathcal{Y}}
\newcommand{\norm}[1]{\left\lVert #1\right\rVert}
\newcommand{\ip}[2]{\left\langle #1,#2\right\rangle}
\newcommand{\op}{\mathrm{op}}
\newcommand{\Gap}{\mathop{\rm Gap}}

\algnewcommand{\Inputs}[1]{%
  \State \textbf{Inputs:}
  \Statex \hspace*{\algorithmicindent}\parbox[t]{.8\linewidth}{\raggedright #1}
}
\algnewcommand{\Initialize}[1]{%
  \State \textbf{Initialize:}
  \Statex \hspace*{\algorithmicindent}\parbox[t]{.8\linewidth}{\raggedright #1}
}
\algdef{SE}[DOWHILE]{Do}{doWhile}{\algorithmicdo}[1]{\algorithmicwhile\ #1}

\theoremstyle{plain}
\newtheorem{teo}{Theorem}
\theoremstyle{plain}

\theoremstyle{plain}
\newtheorem{pro}{Proposition}
\theoremstyle{plain}

\theoremstyle{plain}

\theoremstyle{plain}

\theoremstyle{plain}
\newtheorem{cor}{Corollary}
\theoremstyle{plain}

\theoremstyle{plain}

\theoremstyle{plain}

\theoremstyle{plain}

\theoremstyle{plain}

\theoremstyle{plain}

\theoremstyle{plain}

\theoremstyle{plain}

\theoremstyle{remark}

\theoremstyle{remark}

\theoremstyle{remark}

\theoremstyle{plain}

\theoremstyle{plain}

\theoremstyle{remark}

\title{Certified Residual Quasi-Newton Methods for Distributed Variational Inequalities}
\author{\begin{tabular}{c}
Ewsey R. Obzherin, Roman M. Mozhaev, Alexander V. Gasnikov,\\
Martin Tak\'a\v{c}, Artem D. Agafonov, Dmitry I. Kamzolov
\end{tabular}}
\date{}
\begin{document}
\maketitle\label{article_begin}
\begin{abstract}
Second-order methods for smooth monotone variational inequalities reach the
optimal rate $O(T^{-3/2})$, but a distributed exact Jacobian costs $d$ times
more communication than an operator value.  We show that similarity does part
of the work for free: if the server's Jacobian differs from the global one by
at most $\beta$, using it gives $O(L_1D^3T^{-3/2}+\beta D^2T^{-1})$ at
first-order communication cost.  A quasi-Newton approximation of the residual
Jacobian $\nabla F-\nabla F_1$, built from secants already communicated,
improves the model but cannot remove the $T^{-1}$ term, because any uniform
bound on the Jacobian error leaves it in the rate.  We therefore certify the
surrogate only along the candidate step: one Jacobian--vector product tests it,
and a failed test is reused as an exact correction.  This attains the exact
rate $O(L_1D^3T^{-3/2})$ while transmitting only vectors.  Experiments on
LIBSVM and synthetic instances measure accuracy against communication.
\end{abstract}

\begingroup\sloppy\noindent\textbf{Keywords:} variational inequality, distributed optimization, second-order similarity, quasi-Newton method, Jacobian--vector product, communication complexity\par\endgroup


\par\smallskip\noindent
\begin{minipage}{\linewidth}
\footnotesize\raggedright
Ewsey Romanovich Obzherin\textsuperscript{1,2}, \url{obzherin.ewsey@ya.ru}\par
Roman Mikhailovich Mozhaev\textsuperscript{1}, \url{ceo@mozhaevr.ru}\par
Alexander Vladimirovich Gasnikov\textsuperscript{1,3}, \url{gasnikov@yandex.ru}\par
Martin Tak\'a\v{c}\textsuperscript{4}, \url{martin.takac@mbzuai.ac.ae}\par
Artem Dmitrievich Agafonov\textsuperscript{1,4}, \url{agafonov.artem98@gmail.com}\par
Dmitry Igorevich Kamzolov, Independent Researcher, \url{kamzolov.opt@gmail.com}\par
\smallskip
\textsuperscript{1}\,Moscow Institute of Physics and Technology,
9 Institutskiy Lane, Dolgoprudny, Moscow Region 141700, Russia\par
\textsuperscript{2}\,FusionBrain Lab, Artificial Intelligence Research Institute,
6 Presnenskaya Embankment, Building 2, Floor 4, Moscow 123112, Russia\par
\textsuperscript{3}\,Innopolis University, 1 Universitetskaya Street, Innopolis 420500, Russia\par
\textsuperscript{4}\,Mohamed bin Zayed University of Artificial Intelligence,
P.O. Box 7909, Masdar City, Abu Dhabi, \mbox{United Arab Emirates}\par
\end{minipage}
\par\medskip
{
\extrarowheight=-2pt

\section{Introduction}
\label{sec:introduction}

Many problems in engineering, economics, and machine learning ask for an
\emph{equilibrium} rather than for a minimum: several agents act at once,
and we look for a state in which none of them can improve alone.
Variational inequalities (VIs) are
the standard language for questions of this type~\cite{facchinei2003,kinderlehrer1980}.
They describe traffic equilibria, where every driver picks the fastest
route and the choices of the others change what ``fastest'' means~\cite{dafermos1980};
market models and Nash equilibria, where each firm reacts to the decisions of
its competitors~\cite{facchinei2007}; and power control in wireless
networks~\cite{scutari2010}.  The same form appears in machine learning, for
example in adversarial training~\cite{madry2018} and in generative adversarial
networks~\cite{goodfellow2014}.  VIs also have a long history in control theory.
Frequency-domain conditions for feedback systems go back to
Yakubovich~\cite{yakubovich1962}. The VI formulation considered here
separately requires a monotone equilibrium operator.
Monotonicity is still the main tool in
modern control work on equilibrium seeking over
networks~\cite{tatarenko2021}. Because so many different models share one
structure, a fast and reliable VI solver is useful far beyond a single
application.

Throughout the paper $\Z\subset\R^d$ is a nonempty compact convex set and
$F:\Z\to\R^d$ is a continuous operator.  We call $F$ monotone if
\begin{equation}
    \label{eq:monotone}
  \ip{F(x)-F(y)}{x-y}\ge0
  \qquad\text{for all }x,y\in\Z,
\end{equation}
and $\mu$-strongly monotone if the left-hand side is at least
$\mu\norm{x-y}^2$ for some $\mu>0$.  Monotonicity plays here the role
of convexity: when $F=\nabla f$, the two conditions above
say exactly that $f$ is convex, respectively $\mu$-strongly convex.
The variational inequality problem is
\begin{equation}
  \text{find } z^\star\in\Z
  \quad\text{such that}\quad
  \ip{F(z^\star)}{z-z^\star}\ge0
  \quad\text{for every }z\in\Z.
  \label{eq:vi}
\end{equation}
Geometrically, $-F(z^\star)$ points out of the
set $\Z$, so no feasible direction improves the situation; if $z^\star$ is
interior, \eqref{eq:vi} just means $F(z^\star)=0$.

A solution is almost never reached exactly, so we measure the quality of a point $\widehat z\in\Z$ by the Minty gap
\begin{equation}
  \Gap(\widehat z):=\sup_{z\in\Z}\ip{F(z)}{\widehat z-z}.
  \label{eq:gap}
\end{equation}
For monotone $F$ this quantity is nonnegative on $\Z$, and it equals zero
exactly at a solution of~\eqref{eq:vi}.  All rates in this paper
are stated for the gap.  We write $D:=\sup_{x,y\in\Z}\norm{x-y}$ for
the diameter of $\Z$. $L_0$ for the Lipschitz constant of $F$: $\norm{F(x)-F(y)}\le L_0\norm{x-y}$, and $L_1$ for
the Lipschitz constant of its Jacobian,
\begin{equation}
    \label{eq:jac-lip}
  \norm{\nabla F(x)-\nabla F(y)}_{\op}\le L_1\norm{x-y},
  \qquad x,y\in\Z.
\end{equation}
The first drives the rates of first-order methods, the second those of
second-order methods.

Three standard problems are special cases of~\eqref{eq:vi}.  For
$F=\nabla f$ with $f$ convex and differentiable, \eqref{eq:vi} is the
first-order optimality condition of $\min_{z\in\Z}f(z)$.  For $z=(x,y)$,
$\Z=\X\times\Y$ and
$F(z)=\bigl(\nabla_x\phi(x,y),\,-\nabla_y\phi(x,y)\bigr)$, its solutions are
the saddle points of $\min_{x\in\X}\max_{y\in\Y}\phi(x,y)$, and $F$ is monotone
exactly when $\phi$ is convex--concave; the minus sign makes $F$ a non-gradient
field, so $\nabla F$ is in general not symmetric and BFGS-type formulas, which
produce symmetric matrices, do not apply directly.  For $\Z=\R^d$,
\eqref{eq:vi} reduces to the nonlinear system $F(z^\star)=0$, where $D$ is
infinite and the gap must be replaced by a residual such as
$\norm{F(\widehat z)}$.

\noindent The classical
first-order methods for VIs are Extragradient~\cite{korpelevich1976}, the optimistic (past
extragradient) method~\cite{popov1980}, Mirror--Prox~\cite{nemirovski2004}, and
dual extrapolation~\cite{nesterov2007}.  They use only values of $F$, and they
give the standard ergodic guarantee $\Gap(\widetilde z_T)=O(L_0 D^2T^{-1})$. Curvature
information gives a better rate.  The Newton proximal extragradient method of
Monteiro and Svaiter~\cite{monteiro2012} uses the Jacobian $\nabla F$ and
reaches the classical benchmark $O(L_1D^3 T^{-3/2})$. Related bounds for
Newton and higher-order methods use full derivative
information~\cite{bullins2022,adil2022,lin2025,alves2024}. Faster guarantees
are available for the convex--concave subclass~\cite{chen2025oracle} and,
with a full second-order oracle, for monotone
VIs~\cite[Theorem~5.1]{zhang2026oracle}. Their tangent residual also bounds
the Minty gap after multiplication by $D$. We use the classical rate as our
benchmark for communication. The price of the classical method is
that every step needs a full Jacobian (or a higher-order tensor) and an inner
regularized subproblem.  Quasi-Newton (QN) methods try to avoid this price:
they build an approximation of the Jacobian from operator values
only~\cite{broyden1965,dennismore1977}, and derivative-assisted QN updates
also appear in saddle-point problems~\cite{liu2022}. Agafonov et al.~\cite{agafonov2024} make the
trade-off explicit for VIs: they measure how much Jacobian error a second-order
method can tolerate and use Broyden-type approximations inside it.
Jiang and Mokhtari~\cite{jiang2026qnpe} reuse a rejected trial for the
next outer approximation to a monotone equation. Here a failed test
provides an exact action used before another trial at the same outer point.

In many modern applications the data are not stored on one machine.  This leads
to the distributed (finite-sum) operator
\begin{equation}
  F(z)=\sum_{i=1}^{M}w_iF_i(z),
  \qquad w_i>0,
  \qquad \sum_{i=1}^{M}w_i=1,
  \label{eq:finite-sum}
\end{equation}
where node $i$ can evaluate only its own $F_i$, and the server (which we also
call node~$1$) collects the results.  Here local computation is usually cheap
and communication is expensive, so the right question is not only how many
iterations a method needs, but also how much information must cross the
network.  For first-order methods this question is well understood: there are
lower bounds and matching optimal decentralized algorithms~\cite{kovalev2022},
together with practical schemes that use compression and local
steps~\cite{beznosikov2023}. For second-order methods much less is known, and
the existing Newton-type distributed algorithms were designed mainly for
minimization~\cite{daneshmand2021,islamov2021,safaryan2022}; moreover, one
exact global Jacobian costs $O(Md^2)$ transmitted scalars, that is, $d$ times
more than one global operator value.

This gives a trade-off.  We are willing to compute more on the workers,
because that buys faster convergence, but we are not willing to send a full
Jacobian, because communication is the expensive resource.  Two ideas make
second-order VI methods possible under this constraint.  The first is
quasi-Newton updates, which replace the Jacobian by a matrix assembled from
operator values.  The second is similarity: the operator stored at the server
may already be close to the global one, so its Jacobian is a locally available
model of the global Jacobian.
Distributed quasi-Newton methods with non-asymptotic guarantees exist for
minimization~\cite{du2024,agafonov2025flecs}. Distributed VI methods have
also exploited similarity~\cite{beznosikov2021,beznosikov2023,zhou2024}.
The first-order result of Zhou et al.~\cite{zhou2024} uses an all-component
condition and an expected saddle gap. We use one server's similarity radius
and count full-participation reductions for a deterministic Minty-gap bound.

For minimization, similarity is usually written as
$\norm{\nabla^2f_i(z)-\nabla^2f(z)}_{\op}\le\delta$, where $\delta$ becomes
small when each node's data is statistically close to each other and to the average~\cite{arjevani2015,shamir2014,hendrikx2020}.
Under this assumption the local model
already carries most of the global curvature, so a method can communicate much
less.  It transfers to VIs without loss: for the canonical saddle operator
\[
  F_i(z)=
  \begin{pmatrix}
    \nabla_x\phi_i(x,y)\\
    -\nabla_y\phi_i(x,y)
  \end{pmatrix}
  =D_{\rm sp}\nabla\phi_i(z),
  \qquad
  D_{\rm sp}:=\operatorname{diag}(I,-I),
\]
we have $\nabla F_i(z)=D_{\rm sp}\nabla^2\phi_i(z)$, and $D_{\rm sp}$ is
orthogonal, so Hessian similarity and Jacobian similarity are the same quantity
in operator norm.  We therefore work directly with the Jacobian similarity
parameter
\begin{equation}
  \beta:=\sup_{z\in\Z}
  \norm{\nabla F(z)-\nabla F_1(z)}_{\op}.
  \label{eq:beta}
\end{equation}
If $\beta$ is small, the Jacobian $\nabla F_1$, which the server can compute
without any communication, is already a good model of the global Jacobian; if
$\beta$ is large, this local model is biased in a systematic way.

Bringing the two ideas together in a distributed VI raises three
questions:

\begin{center}
\begin{minipage}{0.92\linewidth}
\centering
\itshape
Can similarity alone communicate less than first-order methods for VIs?

Can a distributed QN method for VIs outperform
first-order methods?

Can distributed QN reach the rate of an exact second-order method cheaper than sending a full Jacobian?
\end{minipage}
\end{center}

\noindent
We answer them in turn.

\paragraph*{Main Contributions.}
\begin{itemize}
\item \textbf{Similarity alone already helps (Section~\ref{sec:similarity}).} If the server uses only its
  own Jacobian $\nabla F_1$, the method converges as
  $O\!\left(L_1D^3T^{-3/2}+\beta D^2T^{-1}\right)$ and sends nothing beyond what
  full-operator extragradient sends per step. The $T^{-1}$ term has a smaller
  coefficient than the standard extragradient bound when $\beta\ll L_0$.

\item \textbf{Residual quasi-Newton (Section~\ref{sec:qn}).} We
  approximate only the difference $\nabla F-\nabla F_1$, using secants the
  method has already communicated, and combine it with the exact server
  Jacobian. A positive uniform bound on the model error enters the VIJI
  guarantee through the $T^{-1}$ term.

\item \textbf{A certificate closes the gap (Section~\ref{sec:method}).} Instead of asking the model
  to be accurate as a matrix, we ask it to be accurate only along the step we
  actually take.  One Jacobian--vector product tests this; if the test fails,
  the same product is reused to correct the model in that direction, and the
  step is recomputed.  The method then attains
  the exact second-order rate
  $O(L_1D^3T^{-3/2})$ while only vectors of length $d$ cross the network.  The
  bound on extra products depends on the prior and the directions visited.
  In the worst case there are at most $d$ products per outer iteration.

\item \textbf{Experiments (Section~\ref{sec:experiments}).} On real and synthetic problems we compare both
  versions with an exact-Jacobian method and with first-order baselines.
\end{itemize}

\textbf{Organization.} Section~\ref{sec:setup} fixes the problem, the
communication model, and the residual operator.
Section~\ref{sec:viji} recalls the inexact-Jacobian theory we build on
and separates its two accuracy requirements.  The three questions are
then answered in order: Section~\ref{sec:similarity} treats similarity
alone, Section~\ref{sec:qn} adds the residual quasi-Newton
approximation, and Section~\ref{sec:method} adds the certificate and
proves the exact rate together with the communication bounds.
Section~\ref{sec:experiments} reports the numerical study.  Proofs are
in the appendices.

\section{Problem Formulation and Communication Model}
\label{sec:setup}

We solve the variational inequality~\eqref{eq:vi} on a nonempty compact convex
set $\Z\subset\R^d$ of diameter $D$, with a continuously differentiable operator
$F$ that is monotone in the sense of~\eqref{eq:monotone} and has an
$L_1$-Lipschitz Jacobian in the sense of~\eqref{eq:jac-lip}.  The output of an
algorithm is judged by the gap~\eqref{eq:gap}.  The operator has the
finite-sum form~\eqref{eq:finite-sum}, and the local--global Jacobian similarity
is measured by $\beta$ from~\eqref{eq:beta}. Choose the initial point
$z_0\in\Z$.

\subsection{Communication model}

The system has $M$ nodes and a server, and the server is also worker~$1$.  It
can therefore evaluate $F_1$, the local Jacobian $\nabla F_1$, and local
Jacobian--vector products without any network exchange.  Everything that
involves the average~\eqref{eq:finite-sum} needs the other nodes.  We assume
full participation: every node answers every request.  Local computation is treated as cheap and communication as expensive, as in federated and cluster settings, so we ask not only how many
iterations a method needs but how many scalars have to cross the network.

Three global quantities appear in second-order methods, and each is one
collective synchronization, so they differ in volume and not in rounds.
An \emph{operator value} has each node send $w_iF_i(v)$, a vector of
length $d$.  A \emph{Jacobian--vector product} has the server broadcast $s$ and
each node return $w_i\nabla F_i(v)s$, again a vector of length $d$; crucially,
no node builds its Jacobian, since $\nabla F_i(v)s$ is a directional derivative
costing about one operator evaluation.  A \emph{full Jacobian} has each node
send the $d^2$ entries of $w_i\nabla F_i(v)$. Under this nominal reduction
accounting the first two cost $Md$ scalar coordinates and the third costs
$Md^2$. These counts exclude broadcasts, local computation, and diagnostics.
The gap of size $d$ between the Jacobian and JVP is the reason for the whole construction of this paper.  

\subsection{Complexity measures}

We count the following quantities separately, because a system may be limited
by any one of them:
\begin{itemize}
\item $N_{\rm sync}$: number of communication rounds;
\item $N_{\rm scal}$: communication volume, measured by the number of
  communicated scalars;
\item $N_F$: number of global operator evaluations;
\item $N_{\rm JVP}$: number of global Jacobian--vector products;
\item $N_J$: number of global Jacobian evaluations.
\end{itemize}
Keeping $N_{\rm sync}$ and
$N_{\rm scal}$ apart matters here: a network with a high fixed cost per
round and plenty of bandwidth would judge the two methods differently from a
network with the opposite profile.

\subsection{The residual operator}

The server already knows $F_1$, so the only unknown part of the problem is the
difference between the average and the server's own operator.  We call it the
residual operator,
\begin{equation}
  G(z):=F(z)-F_1(z),
  \label{eq:G}
\end{equation}
and write its Jacobian at an outer point $v_k$ as
\begin{equation}
  R_k:=\nabla G(v_k)=\nabla F(v_k)-\nabla F_1(v_k).
  \label{eq:Rk}
\end{equation}
This is the object our quasi-Newton matrices approximate.  Approximating $G$
rather than $F$ is what makes similarity useful: by~\eqref{eq:beta} the
quantity we have to learn is small when the nodes hold similar data. The
global Jacobian can still be large even when $\beta=0$.

Since $\norm{\nabla G(z)}_{\op}\le\beta$ on $\Z$ by~\eqref{eq:beta}, the
residual operator is $\beta$-Lipschitz,
\begin{equation}
  \norm{G(x)-G(y)}\le\beta\norm{x-y},
  \qquad x,y\in\Z.
  \label{eq:G-lip}
\end{equation}
This controls the size of residual secants. Their accuracy at a different
point also depends on how much the residual Jacobian varies between the points.

For the error bounds based on stored secants we use one additional assumption,
which is not needed for the basic certified convergence theorem:
\begin{equation}
  \norm{\nabla G(x)-\nabla G(y)}_{\op}
  \le L_G\norm{x-y},
  \qquad x,y\in\Z.
  \label{eq:residual-jac-lip}
\end{equation}
This holds whenever both $\nabla F$ and $\nabla F_1$ are
Lipschitz; it only quantifies how fast an old secant
becomes a poor description of $R_k$.

Finally, a residual secant costs nothing extra: $y=G(b)-G(a)$ combines two
global operator values the outer method has already requested with two local
evaluations at the server, so the quasi-Newton model needs no communication of
its own.

\section{Inexact Jacobians: two kinds of accuracy}
\label{sec:viji}

The local-Jacobian and uncertified residual-QN methods below use the
inexact-Jacobian analysis of Agafonov et al.~\cite{agafonov2024}, which we call
VIJI. It distinguishes uniform and directional accuracy. For the outer
iteration put $b_0=0$ and $v_k=\Pi_{\Z}(z_0+b_k)$. An accepted local solution
$x_k$ defines $s_k=x_k-v_k$ and the update $b_{k+1}=b_k-\lambda_k F(x_k)$.
With positive weights the output is
$\widetilde z_T=\sum_{k<T}\lambda_kx_k/\sum_{k<T}\lambda_k$.
Appendix~\ref{app:mapping} specifies the local model and its stopping test;
on a zero-gap anchor the method returns the solution directly.

\begin{teo}[{\cite[Theorems 3.2 and 3.3]{agafonov2024}}]
\label{thm:known-uniform}
Suppose that the Jacobian surrogate $J_k$ satisfies the uniform bound
\begin{equation}
  \norm{\nabla F(v_k)-J_k}_{\op}\le\delta
  \label{eq:uniform-condition}
\end{equation}
at every outer iteration. Solve the regularized inner VI to
\eqref{eq:app-inner} with the same $\delta$, and set the outer
parameter $\rho_k=\delta$. Then for nonzero accepted steps use
$\lambda_k(L_1\norm{s_k}/2+\rho_k)\in[1/32,1/22]$, and
\begin{equation}
  \Gap(\widetilde z_T)
  \le\frac{16\sqrt2\,L_1D^3}{T^{3/2}}+\frac{16\sqrt2\,\delta D^2}{T}.
  \label{eq:known-rate}
\end{equation}
The directional case in the cited analysis uses a trial-dependent
inner-model level $\delta_k$ satisfying
$\norm{(\nabla F(v_k)-J_k)s_k}\le\delta_k\norm{s_k}$ and
$\delta_k\le L_1\norm{s_k}/2$. In particular it enforces
\begin{equation}
  \norm{(\nabla F(v_k)-J_k)s_k}\le\frac{L_1}{2}\norm{s_k}^2,
  \label{eq:directional-condition}
\end{equation}
where on a lower-dimensional set the certified version uses the
restriction and projection specified in Section~\ref{sec:method}.
Using $\rho_k=L_1\norm{s_k}/2$ and the corresponding outer weights,
its second-order bound is
\begin{equation}
  \Gap(\widetilde z_T)\le\frac{32L_1D^3}{T^{3/2}}.
  \label{eq:known-exact-rate}
\end{equation}
\end{teo}

The uniform bound \eqref{eq:uniform-condition} holds in every direction
and ensures admissibility of the inner model. The resulting upper bound
contains a $T^{-1}$ term when $\delta>0$. The directional test
\eqref{eq:directional-condition} only concerns the actual step. A rate
for the certified algorithm also needs control of the level used in the
local model and its stopping test. We prove the corresponding bound
for our algorithm in Appendix~\ref{app:proofs}.

\section{Similarity alone: the local-Jacobian baseline}
\label{sec:similarity}

The first question asks how far similarity gets us on its own.  The answer is
immediate from Theorem~\ref{thm:known-uniform}: the server uses its own
Jacobian as the surrogate and pays nothing for it.

\begin{cor}[Local-Jacobian rate]
\label{cor:local}
Set $J_k=\nabla F_1(v_k)$.  By \eqref{eq:beta} the uniform condition
\eqref{eq:uniform-condition} holds with $\delta=\beta$, so
\begin{equation}
  \Gap(\widetilde z_T)
  \le
  \frac{16\sqrt2\,L_1D^3}{T^{3/2}}
  +\frac{16\sqrt2\,\beta D^2}{T}.
  \label{eq:local-rate}
\end{equation}
\end{cor}

We call \eqref{eq:local-rate} the \emph{local-Jacobian rate} and the method attaining it
the local-Jacobian baseline.

\begin{cor}[Communication of the local-Jacobian baseline]
\label{cor:local-comm}
To reach $\Gap\le\varepsilon$ the local-Jacobian baseline needs at most
\[
  O\!\left(1+\left(\frac{L_1D^3}{\varepsilon}\right)^{2/3}
      +\frac{\beta D^2}{\varepsilon}\right)
\]
outer iterations. Each accepted nonzero iteration uses two reductions and
transmits $2Md$ scalars. An exact-solution anchor check may terminate the
method earlier. No Jacobian products or full Jacobians are transmitted.
\end{cor}

This answers the first question at the level of upper bounds. Applying
Extragradient to the full operator gives the standard
$O(L_0D^2/\varepsilon)$ iteration bound with the same two vector reductions
per iteration. The coefficient of the $\varepsilon^{-1}$ term in the
local-Jacobian bound is smaller when $\beta\ll L_0$.

For $\beta>0$ the $\beta/T$ term dominates the local-Jacobian upper bound
past a horizon of order
\[
  T\asymp\left(\frac{L_1D}{\beta}\right)^{2}.
\]
This is a comparison of the two terms in the bound; individual instances
can converge faster.

\section{Residual quasi-Newton}
\label{sec:qn}

The second question asks whether a quasi-Newton approximation improves on the
baseline.  We first say precisely what is approximated.

\subsection{What the quasi-Newton matrix approximates}

The server knows $\nabla F_1$ exactly, so there is no reason to approximate the
whole Jacobian.  We approximate only the residual Jacobian $R_k=\nabla G(v_k)$
of Section~\ref{sec:setup}.  If $B_k$ is such an approximation, the surrogate is
\begin{equation}
  J_k^{\rm QN}:=\nabla F_1(v_k)+B_k,
  \label{eq:combined-qn}
\end{equation}
and the resulting Jacobian error is
\begin{equation}
  H_k:=R_k-B_k
  =\nabla F(v_k)-J_k^{\rm QN}.
  \label{eq:Hk}
\end{equation}
The identity \eqref{eq:Hk} is exact: the server part contributes nothing to the
error, and all inexactness of the surrogate is the residual-QN error $H_k$.
This is the sense in which the exact local Jacobian and the quasi-Newton
approximation are combined.

Section~\ref{sec:setup} makes $R_k$ the right target: it is small under
similarity, and its secants cost no communication.

\subsection{Building $B_k$ from stored residual secants}

For residual secant pairs
\[
  s_j=b_j-a_j,
  \qquad
  y_j=G(b_j)-G(a_j),
  \qquad j=1,\dots,m,
\]
we have $\norm{y_j}\le\beta\norm{s_j}$ by \eqref{eq:G-lip}.  Collect them into
\[
  S=[s_1,\ldots,s_m],
  \qquad
  Y=[y_1,\ldots,y_m],
\]
and define the minimum-Frobenius-norm least-squares secant approximation
\begin{equation}
  \widehat B_k:=YS^\dagger,
  \label{eq:history-prior}
\end{equation}
where $S^\dagger$ is the Moore--Penrose pseudoinverse. The matrix
$YS^\dagger$ is the minimum-Frobenius-norm minimizer of $\norm{BS-Y}_{\mathrm F}$.
If $S$ has full column
rank, then $S^\dagger S=I_m$ and $\widehat B_kS=Y$, so $\widehat B_k$ is an
exact multisecant matrix~\cite{fang2009}; for rank-deficient $S$ it is the
minimum-norm least-squares solution.  No rank assumption on $S$ is needed
anywhere below, and no separate approximation of an individual worker operator
$F_i$ is ever formed.

Similarity gives one piece of information that a general quasi-Newton problem
does not have: the unknown target itself lies in a known set,
\[
  R_k\in\mathcal B_\beta,
  \qquad
  \mathcal B_\beta:=\{B\in\R^{d\times d}:\norm{B}_{\op}\le\beta\}.
\]
There is no reason to return a matrix outside this set, so we use the projected
approximation
\begin{equation}
  B_k:=\Pi_{\mathcal B_\beta}^{\mathrm F}\bigl(\widehat B_k\bigr),
  \label{eq:projected-history-prior}
\end{equation}
where $\Pi^{\mathrm F}$ is the metric projection in the Frobenius norm,
computationally a clipping of the singular values at $\beta$.

\begin{pro}[Uniform bound from similarity]
\label{prop:qn-envelope}
For the projected approximation \eqref{eq:projected-history-prior},
\begin{equation}
  \norm{R_k-B_k}_{\op}\le 2\beta
  \qquad\text{for every }k,
  \label{eq:projected-residual-error}
\end{equation}
and the projection never increases the Frobenius error:
\begin{equation}
  \norm{R_k-B_k}_{\mathrm F}
  \le
  \norm{R_k-\widehat B_k}_{\mathrm F}.
  \label{eq:projection-nonexpansive}
\end{equation}
\end{pro}

Both statements are immediate: $R_k$ and $B_k$ both lie in $\mathcal B_\beta$,
which gives \eqref{eq:projected-residual-error} by the triangle inequality, and
$R_k\in\mathcal B_\beta$ makes the projection nonexpansive relative to $R_k$.
The threshold is thus fixed by the similarity assumption itself.  The
projection is not needed to run the method; it is what turns similarity into a
bound on $\norm{H_k}$ that holds whatever the stored secants look like.

\subsection{How good is a model built from old secants?}

A secant measured on an earlier segment describes the current residual Jacobian
only up to the change of curvature between the two points.  Write
\[
  P_S:=SS^\dagger
\]
for the orthogonal projector onto the span of the stored directions.

\begin{teo}[Error of a model built from stored secants]
\label{thm:recycled-secants}
Assume \eqref{eq:residual-jac-lip}.  The unprojected matrix
\eqref{eq:history-prior} satisfies
\begin{equation}
  R_k-\widehat B_k
  =R_k(I-P_S)+(R_kS-Y)S^\dagger,
  \label{eq:history-decomposition}
\end{equation}
and the two terms are Frobenius-orthogonal, so
\begin{equation}
  \norm{R_k-\widehat B_k}_{\mathrm F}^2
  =\norm{R_k(I-P_S)}_{\mathrm F}^2
  +\norm{(R_kS-Y)S^\dagger}_{\mathrm F}^2.
  \label{eq:history-pythagoras}
\end{equation}
For the projected matrix \eqref{eq:projected-history-prior},
\begin{equation}
\begin{split}
  \norm{R_k-B_k}_{\mathrm F}^2
  \le\;&
  \norm{R_k(I-P_S)}_{\mathrm F}^2\\
  &+L_G^2\norm{S^\dagger}_{\op}^2
  \sum_{j=1}^{m}
  \left(\norm{v_k-a_j}+\frac12\norm{s_j}\right)^2
  \norm{s_j}^2.
\end{split}
\label{eq:history-bound}
\end{equation}
\end{teo}

The decomposition separates two sources of error that behave
differently.  The
first term is the part of the residual Jacobian acting in directions the stored
secants do not cover; it is reduced by collecting more and more
diverse secants.  The second term is the price of using secants measured away from
$v_k$; it grows with the distance $\norm{v_k-a_j}$ and is amplified when $S$ is
poorly conditioned. Many nearly parallel secants can therefore give a worse model than a few
well-spread ones.

Products stored from earlier certificate tests admit the same kind of bound, which lets the certified method of the next section to reuse them.

\begin{cor}[Error of a model built from stored products]
\label{cor:recycled-jvp}
Suppose $u_1,\ldots,u_m$ are orthonormal and that at earlier points $v_j$ the
algorithm stored the exact products $h_j=\nabla G(v_j)u_j$.  Put
\[
  U=[u_1,\ldots,u_m],
  \quad P_U=UU^{\mathsf T},
  \quad
  \widehat B_k^{\rm jvp}:=\sum_{j=1}^{m}h_ju_j^{\mathsf T},
  \quad
  B_k^{\rm jvp}:=\Pi_{\mathcal B_\beta}^{\mathrm F}\bigl(\widehat B_k^{\rm jvp}\bigr).
\]
Then, under \eqref{eq:residual-jac-lip},
\begin{equation}
  \norm{R_k-B_k^{\rm jvp}}_{\mathrm F}^2
  \le
  \norm{R_k(I-P_U)}_{\mathrm F}^2
  +L_G^2\sum_{j=1}^{m}\norm{v_k-v_j}^2.
  \label{eq:recycled-jvp-bound}
\end{equation}
\end{cor}

\subsection{What residual quasi-Newton achieves, and what it does not}

Putting \eqref{eq:combined-qn} into Theorem~\ref{thm:known-uniform} gives the
answer to the second question.

\begin{pro}[Uncertified rate]
\label{prop:combined-rate}
Suppose $\norm{H_k}_{\op}\le\delta_{\rm QN}$ at every outer iteration and that
$J_k^{\rm QN}$ is used without a directional certificate.  Then
\begin{equation}
  \Gap(\widetilde z_T)
  \le
  \frac{16\sqrt2\,L_1D^3}{T^{3/2}}
  +\frac{16\sqrt2\,\delta_{\rm QN}D^2}{T}.
  \label{eq:combined-rate}
\end{equation}
For the projected approximation \eqref{eq:projected-history-prior},
Proposition~\ref{prop:qn-envelope} gives $\delta_{\rm QN}=2\beta$, hence
\[
  \Gap(\widetilde z_T)
  \le
  \frac{16\sqrt2\,L_1D^3}{T^{3/2}}
  +\frac{32\sqrt2\,\beta D^2}{T}.
\]
\end{pro}

The method still uses the same two vector reductions per iteration as the
local-Jacobian baseline. A reduction in the Jacobian error can improve the
uniform upper bound with the same per-iteration payload.

The answer is also incomplete.  The guaranteed bound
$\delta_{\rm QN}\le2\beta$ is, in the worst case, no better than the baseline
value $\beta$; what a good model really buys is a smaller $\norm{H_k}$ at the
iterations where the stored secants are informative, which
Theorem~\ref{thm:recycled-secants} quantifies.  A theorem that sees only a
uniform bound cannot use this variation, since a fixed positive $\delta_{\rm QN}$
contributes a $T^{-1}$ term to the displayed estimate.

\section{Directional certificates and the exact second-order rate}
\label{sec:method}

The third question asks for the exact second-order rate without the cost of a
full Jacobian. The method checks the directional condition
\eqref{eq:directional-condition} by one global Jacobian--vector product and
refines only when the check fails. Appendix~\ref{app:proofs} proves the rate
for the resulting local models.

For the projected approximation, Proposition~\ref{prop:qn-envelope} gives
\begin{equation}
  \norm{H_k}_{\op}\le\bar\delta,\qquad \bar\delta:=2\beta.
  \label{eq:uniform-envelope}
\end{equation}
Here $H_k$ is the error of the prior $J_k^{\rm QN}$. The value $\bar\delta$
can be used as a cap in choosing the model level. The rate below uses a
step-dependent bound for the level actually used in the inner solve.

\subsection{Refinement}

Let $E=\operatorname{span}(\Z-\Z)$, $\Pi_E$ be its orthogonal projector,
and $m_E=\dim E$. Operations on $E$ use orthonormal coordinates, with
$\Pi_E=I$ in full dimension. Projecting the operator onto $E$ leaves all
VI pairings unchanged. Fix an outer iteration $k$ and set
$A_k=\Pi_E\nabla F(v_k)|_E$ and $J_k^0=\Pi_EJ_k^{\rm QN}|_E$. Write
$\operatorname{sym}(J)=(J+J^{\mathsf T})/2$ and
$\operatorname{skew}(J)=(J-J^{\mathsf T})/2$. Use
\[
 J_k^+=\operatorname{skew}(J_k^0)
       +\Pi_{\succeq0}^{\mathrm F}\bigl(\operatorname{sym}(J_k^0)\bigr),
\]
where the projection clips negative eigenvalues to zero.

After $r$ failed tests, $U_{k,r}$ holds the learned orthonormal directions
and $W_{k,r}$ holds their global Jacobian actions in ambient coordinates.
Put $P_{k,r}:=U_{k,r}U_{k,r}^{\mathsf T}$ and $Q_{k,r}:=I_E-P_{k,r}$.
We retain the learned residual correction $C_{k,r}:=H_kP_{k,r}$ as a map
from $E$ to $\R^d$. Both $U_{k,0}$ and $W_{k,0}$ are empty and
$P_{k,0}=C_{k,0}=0$. Set $W^E_{k,r}:=\Pi_EW_{k,r}$. The surrogate is
\begin{equation}
  J_{k,r}=W^E_{k,r}U_{k,r}^{\mathsf T}
    -U_{k,r}(W^E_{k,r})^{\mathsf T}Q_{k,r}+Q_{k,r}J_k^+Q_{k,r}.
  \label{eq:Jkr}
\end{equation}
It interpolates the learned directions and is monotone:
\begin{equation}
  W_{k,r}=\nabla F(v_k)U_{k,r},\qquad
  J_{k,r}U_{k,r}=W^E_{k,r},\qquad
  \operatorname{sym}(J_{k,r})\succeq0.
  \label{eq:refinement-invariant}
\end{equation}
In the implementation, the first term uses
$W^E_{k,r}=\Pi_E(J_k^{\rm QN}U_{k,r}+C_{k,r}U_{k,r})$. The
full actions can also be stored directly. The simultaneous requirements of
interpolation and monotonicity are used below for finite local refinement.

For $s=x-v_k$ the server solves the local model and gap test
\eqref{eq:app-model}--\eqref{eq:app-inner} with $\Pi_EF(v_k)$, $J_{k,r}$
and a nonnegative trial level $\delta$. For the step-dependent variant it
uses a level $\delta^{\rm dir}_{k,r}$ and candidate that pass both the
inner test and the additional scale test
\[
  \delta^{\rm dir}_{k,r}\le\frac{L_1}{2}\norm{s_{k,r}}.
\]
One may try positive levels locally and use zero as a fallback. For $L_1>0$ the
model is strictly monotone on $\Z$ since $J_{k,r}$ is monotone and
$s\mapsto\norm{s}s$ is the gradient of the strictly convex function
$\norm{s}^3/3$. Thus every trial model has a unique exact VI solution.
If the anchor is already a solution, the server can detect this from
$F(v_k)$ by the zero-gap test with $x=v_k$. Otherwise the exact local solution
at level zero has $s\ne0$, so its inner tolerance is positive and a convergent
local solver reaches that tolerance in finite time. A local level search
performs no JVPs. The fixed choice $\delta=\bar\delta$ can also be used
to compare model regularizations in experiments. That comparison satisfies
the directional test, but it is outside the step-dependent rate theorem.

\begin{algorithm}[t]
\caption{Certified residual-QN refinement at outer iteration $k$}
\label{alg:certified-refinement}
\begin{algorithmic}[1]
\Inputs{outer point $v_k$; stored residual pairs $(S,Y)$;
  constants $L_1>0$ and $\beta$, with cap $\bar\delta=2\beta$.}
\State \textbf{Server.} If $F(v_k)$ certifies the anchor as an exact solution,
  return it and terminate the outer method.
\State Form $\widehat B_k=YS^\dagger$, or set $\widehat B_k=0$ if the memory
  is empty, and compute $B_k$ by \eqref{eq:projected-history-prior}.
\State Compute $J_k^+$. Set $U_{k,0},W_{k,0}$ empty, $P_{k,0}=C_{k,0}=0$, $r=0$.
\While{true}
  \State Form $J_{k,r}$ by \eqref{eq:Jkr}. Solve \eqref{eq:app-inner} and the
    model-level test locally, using level zero if needed.
  \State Let $x_{k,r}$ be the resulting candidate and set $s_{k,r}=x_{k,r}-v_k$.
  \State Set $w_{k,r}=Q_{k,r}s_{k,r}$.
  \If{$w_{k,r}=0$}
    \State Set $\rho_k=L_1\norm{s_{k,r}}/2$ and \Return $x_{k,r}$.
  \EndIf
  \State \textbf{Communication.} Broadcast $s_{k,r}$ and reduce the weighted
    actions $g_i=w_i\nabla F_i(v_k)s_{k,r}$ to $g$.
  \State \textbf{Server.} Compute $e_{k,r}=\Pi_Eg-J_{k,r}s_{k,r}$.
  \If{$\norm{e_{k,r}}\le L_1\norm{s_{k,r}}^2/2$}
    \State Set $\rho_k=L_1\norm{s_{k,r}}/2$ and \Return $x_{k,r}$.
  \EndIf
  \State Set $u_{k,r+1}=w_{k,r}/\norm{w_{k,r}}$ and
    $a=(g-W_{k,r}U_{k,r}^{\mathsf T}s_{k,r})/\norm{w_{k,r}}$.
  \State Append $u_{k,r+1}$ to $U_{k,r}$ and $a$ to $W_{k,r}$.
    Update $P_{k,r+1}=P_{k,r}+u_{k,r+1}u_{k,r+1}^{\mathsf T}$.
  \State Set $h_{k,r+1}=a-J_k^{\rm QN}u_{k,r+1}$ and
    $C_{k,r+1}=C_{k,r}+h_{k,r+1}u_{k,r+1}^{\mathsf T}$, then increase $r$.
\EndWhile
\end{algorithmic}
\end{algorithm}

A failure gives the exact action $a=\nabla F(v_k)u_{k,r+1}$ in a
new direction. The action $h_{k,r+1}+B_ku_{k,r+1}=R_ku_{k,r+1}$
can be kept as a stored residual product for subsequent outer iterations.
At full rank $Q_{k,r}=0$, so the next certified local candidate needs
no additional JVP. When $L_1=0$, $F$ is affine on $\Z$. Its restriction is
recovered with $m_E$ independent JVPs together with $F(v_0)$, and the affine
VI is solved locally. If $m_E=0$, the feasible set is a singleton.

\subsection{The exact rate}

Use $b_0=0$ and $v_k=\Pi_{\Z}(z_0+b_k)$. For a nonzero accepted step
$s_k=x_k-v_k$, use
\[
  \rho_k=\frac{L_1}{2}\norm{s_k},\qquad
  \lambda_k=\frac{1}{27(L_1\norm{s_k}/2+\rho_k)},\qquad
  b_{k+1}=b_k-\lambda_k F(x_k).
\]
The output after $T$ accepted steps is
$\widetilde z_T=\sum_{k<T}\lambda_kx_k/\sum_{k<T}\lambda_k$.
An exact solution found by the anchor test is returned directly.

\begin{teo}[Certified distributed refinement]
\label{thm:main}
Assume \eqref{eq:monotone}, \eqref{eq:jac-lip}, and $L_1>0$. At each outer
iteration build $B_k$ by \eqref{eq:projected-history-prior} and run
Algorithm~\ref{alg:certified-refinement}. In exact arithmetic every accepted
step satisfies \eqref{eq:directional-condition}, and after $T$ nonzero
accepted steps
\begin{equation}
  \Gap(\widetilde z_T)
  \le\frac{32L_1D^3}{T^{3/2}}.
  \label{eq:main-rate}
\end{equation}
The refinement uses at most $m_E\le d$ global Jacobian--vector
products per outer iteration. The uniform similarity bound is used
for the choice and analysis of the prior. The rate holds for any finite
prior processed by the same completion.
\end{teo}

This recovers the second-order rate of \eqref{eq:known-exact-rate} with
$N_J=0$. The proof uses the inner, level and directional tests in the
same trial, without a fixed positive model level.

\subsection{How many products the certificate costs}

The bound $q_k\le d$ is a worst-case guarantee. To measure the effect of the
initial model error and the monotone completion, define
\[
 \mathcal C_k:=4\sum_{r\in\mathcal F_k}
      \norm{P_{k,r}\operatorname{sym}(A_k)u_{k,r+1}}^2,
\]
where $\mathcal F_k$ is the set of failed directional tests at iteration $k$.

\begin{teo}[Bound on failed certificates]
\label{thm:instance}
Let $s_{k,r}$ be the candidate step of trial $r$ and let $q_k$ be the number
of global Jacobian--vector products used. Then
\begin{equation}
  \frac{L_1^2}{4}\sum_{r\in\mathcal F_k}\norm{s_{k,r}}^2
  \le\norm{H_k}_{\mathrm F}^2+\mathcal C_k,
  \qquad 0\le\mathcal C_k\le2\norm{\operatorname{sym}(A_k)}_{\mathrm F}^2.
  \label{eq:energy-budget}
\end{equation}
If no test fails, $q_k\le1$. Otherwise put
$\tau_k:=\min_{r\in\mathcal F_k}\norm{s_{k,r}}>0$. Then
\begin{equation}
  q_k\le\min\!\left\{d,\;
  1+\frac{4\bigl(\norm{H_k}_{\mathrm F}^2+\mathcal C_k\bigr)}{L_1^2\tau_k^2}
  \right\}.
  \label{eq:q-bound}
\end{equation}
\end{teo}

Each failure yields one orthogonal direction. Its contribution to the bound
is split between the prior Frobenius error and the symmetric block coupling
generated by the monotone completion. These terms bound the number of long
failed steps. In particular the prior term can be written as
$\operatorname{sr}(H_k)\norm{H_k}_{\op}^2$ when $H_k\ne0$. The additional
term $\mathcal C_k$ and the scale $\tau_k$ depend on the actual refinement path.

Combining Theorem~\ref{thm:instance} with Theorem~\ref{thm:recycled-secants}
expresses the prior contribution in terms of the stored secants.

\begin{cor}[Product count from stored secants]
\label{cor:history-jvp}
If a test fails and $B_k$ is built as in Theorem~\ref{thm:recycled-secants}, then
\begin{equation}
\begin{split}
  q_k\le\min\Bigg\{d,\;1+\frac{4}{L_1^2\tau_k^2}\Bigg[
  &\mathcal C_k+\norm{R_k(I-P_S)}_{\mathrm F}^2\\
  &+L_G^2\norm{S^\dagger}_{\op}^2\sum_{j=1}^{m}
   \left(\norm{v_k-a_j}+\tfrac12\norm{s_j}\right)^2\norm{s_j}^2
  \Bigg]\Bigg\}.
\end{split}
\label{eq:history-q-bound}
\end{equation}
\end{cor}

The secant estimate controls the prior-error part of the certificate cost.
Both $\mathcal C_k$ and $\tau_k$ may change with the chosen prior, so this
bound does not order the communication of different priors.

\subsection{End-to-end communication}

For $D,L_1,\varepsilon>0$, let
\begin{equation}
  T_\varepsilon
  :=\left\lceil
      \left(\frac{32L_1D^3}{\varepsilon}\right)^{2/3}
    \right\rceil,
  \qquad
  Q_T:=\sum_{k=0}^{T-1}q_k,
  \label{eq:T-eps}
\end{equation}
the iterations needed for accuracy $\varepsilon$ by \eqref{eq:main-rate}
and the total number of global products over $T$ iterations. We count two
global operator reductions per accepted iteration. Unvisited iterations
contribute zero to $Q_{T_\varepsilon}$ after an exact-solution return.

\begin{teo}[End-to-end communication complexity]
\label{thm:communication}
The certified method reaches $\Gap\le\varepsilon$ with at most the following costs:
\begin{align}
  N_{\rm sync}^{\rm cert}
  &\le2T_\varepsilon+Q_{T_\varepsilon},
  &
  N_{\rm scal}^{\rm cert}
  &\le Md\bigl(2T_\varepsilon+Q_{T_\varepsilon}\bigr),
  \label{eq:cert-comm}\\
  N_F^{\rm cert}&\le2T_\varepsilon,
  &
  N_{\rm JVP}^{\rm cert}&\le Q_{T_\varepsilon},
  \qquad N_J^{\rm cert}=0,
  \label{eq:cert-oracles}
\end{align}
and
\begin{equation}
  Q_T
  \le
  \min\!\left\{
    dT,\;
    T+\frac{4}{L_1^2}
    \sum_{\substack{0\le k<T\\ \mathcal F_k\neq\varnothing}}
    \frac{\norm{H_k}_{\mathrm F}^2+\mathcal C_k}{\tau_k^2}
  \right\}.
  \label{eq:pathwise-Q}
\end{equation}
At the same accepted-step horizon the method using exact Jacobians has cost
\begin{equation}
  N_{\rm sync}^{\rm exact}=3T_\varepsilon,
  \qquad
  N_{\rm scal}^{\rm exact}=M(2d+d^2)T_\varepsilon,
  \label{eq:exact-comm}
\end{equation}
so that
\begin{equation}
  N_{\rm scal}^{\rm cert}\le N_{\rm scal}^{\rm exact},
  \label{eq:payload-domination}
\end{equation}
with strict inequality if $Q_{T_\varepsilon}/T_\varepsilon<d$.
\end{teo}

If the model is built from stored secants, $\norm{H_k}_{\mathrm F}^2$ in
\eqref{eq:pathwise-Q} may be replaced by the right-hand side of
\eqref{eq:history-bound}.

\begin{cor}[Bandwidth reduction factor]
\label{cor:bandwidth-factor}
For a common horizon $T\ge1$ of accepted steps,
\begin{equation}
  \frac{N_{\rm scal}^{\rm cert}}{N_{\rm scal}^{\rm exact}}
  =\frac{2+Q_T/T}{d+2}.
  \label{eq:bandwidth-factor}
\end{equation}
Hence in any regime with $Q_T/T=o(d)$ the certified method transmits a
vanishing fraction of the scalars a full-Jacobian method needs, as $d$ grows.
\end{cor}

Volume and rounds must be kept apart.  The exact-Jacobian method uses
three collective phases per iteration, the certified method $2+q_k$, and its
extra products are sequential, since each depends on the candidate produced
after the previous correction.  The certified method can therefore transmit far
less data and still use more rounds; which matters depends on the network.

\begin{table}[t]
\caption{Communication per accepted outer iteration.  The two operator
reductions common to all four methods are included.}
\label{tab:per-iteration}
\centering
\small
\begin{tabular}{lccc}
\hline
Method & Synchronizations $\downarrow$ & Scalars sent $\downarrow$ & Curvature source\\
\hline
Local Jacobian & $2$ & $2Md$ & local Jacobian\\
Residual QN & $2$ & $2Md$ & stored residual secants\\
Exact-J & $3$ & $M(2d+d^2)$ & full global Jacobian\\
Certified residual QN & $2+q_k$ & $Md(2+q_k)$ & $q_k$ sequential products\\
\hline
\end{tabular}
\end{table}

\section{Numerical study}
\label{sec:experiments}
\newcommand{\NumericalProvenance}{The source protocol follows revision \texttt{4271b92d} of the experiment repository. \par\noindent\url{https://github.com/mojaevr/VI-Similarity-vs-QN-RJND-}\par\noindent The accompanying source bundle includes the certificate patch, the raw logs and the commands for \texttt{rerun\_revision.py} and \texttt{prepare\_assets.py}.}
\newcommand{\NumericalFigRefusal}{The first-order and local-only baselines are plotted up to the target. Crosses mark termination without an accepted new iterate.}
\newcommand{\NumericalTableCounts}{A fraction next to an entry gives the number of successful runs when some did not reach the target. Means use successful runs only.}
\newcommand{\NumericalAblation}{The auxiliary exact-Jacobian level sweep from \texttt{4271b92d}, retained as an archival diagnostic, varied the supplied level while using the same Jacobian. Its instance-dependent trend motivates keeping the fixed-level control.}
\newcommand{\NumericalPanelBRefusal}{Crosses give the final cost of runs that stopped before the target.}
\newcommand{\NumericalInnerProtocol}{The main runs allow 600 extragradient iterations initially and double the budget up to 4000 when needed. The separate rate-study budget appears below.}
\newcommand{\NumericalFailureProtocol}{At a failed inner solve or an exhausted certificate the run ends without updating its accepted trajectory.}
\newcommand{\NumericalConstantProtocol}{For the logistic instances, a valid upper bound is $L_1=1.6$: with $t=b_jw^{\mathsf T}(a_j+r)$ one has $\|\nabla t\|\le\sqrt{41}/4$, $\|\nabla^2t\|_{\rm op}\le1$, $\ell''\le1/4$ and $|\ell'''|\le1/(6\sqrt3)$. The chain rule gives a bound below $1.6$ for each Jacobian, and averaging preserves it. In a separate rerun at this bound on three representative real-data partitions, 18 of 24 runs reached the common target within the original outer limits.}
\newcommand{\NumericalComparison}{Table~\ref{tab:tolerance} gives the empirical comparison at a common accuracy. On \texttt{a9a} at $h=0$, Certified $+$ QN with $\bar\delta$ reaches the target with mean payload $5.50\!\cdot\!10^{5}$ in 25.8 accepted steps. The corresponding mean payloads of Exact-J and Local-J are $5.03\!\cdot\!10^{7}$ and $3.01\!\cdot\!10^{6}$.}
\newcommand{\NumericalPanelObservation}{At $h=0$ the seed-0 counts at the target are 37 and 37, respectively. At $h=1$ the seed-0 counts at the target are 29 and 29, respectively.}
\newcommand{\NumericalRateObservation}{The panels use the original instance, stopping range and nominal scales. For $\beta=0.4$, the shortest accepted prefixes across the three replicas are 144 (Exact-J), 257 (Local-J), 400 (Certified $+$ QN, $\bar\delta$), and 102 (Certified $+$ QN, $\delta_k^{\rm dir}$).}
\newcommand{\NumericalStress}{Exact-J reaches the target in 9 accepted steps. Extragradient reaches the target in 2432 accepted steps. Certified $+$ QN, $\bar\delta$ ends after 600 accepted steps without reaching the target. Certified $+$ QN, $\delta_k^{\rm dir}$ terminates before the first accepted step.}

\noindent
The study asks when communicating curvature reduces total cost, when
Jacobian--vector products replace a full Jacobian at a smaller nominal
payload, and how the two model levels affect the measured trajectories.
Appendix~\ref{app:reproducibility} carries the auxiliary instances, protocol,
and secondary measurements.
\NumericalProvenance

\paragraph{Instance and methods.}
\label{sec:instances}
The instance is distributed robust logistic regression with a shared adversarial
perturbation,
\begin{equation}
  \min_{\norm{w}\le R_w}\max_{\norm{r}\le R_r}\;
  \tfrac1N\textstyle\sum_{j=1}^{N}
  \log\bigl(1+e^{-b_jw^{\mathsf T}(a_j+r)}\bigr)-\tfrac{\gamma}{2}\norm{r}^2 ,
  \label{eq:instance}
\end{equation}
on \texttt{a9a} (VI dimension $d=246$) and \texttt{w8a} ($d=600$) over $M=25$ nodes, with
$R_w=1$, $R_r=1/4$ and $\gamma=R_w^2/4$. We use monotonicity of the
instance. Its Jacobian is nonsymmetric and $z$-dependent. Class-balanced
feature shards indexed by $h\in[0,1]$ set the
heterogeneity: \texttt{a9a} uses five levels with five seeds, and
\texttt{w8a} uses $h=0$. We measure the checkpoint similarity radius
$\widehat\beta_{\rm sim}$ of \eqref{eq:beta-hat} on each partition.

The second-order methods share the outer update, starting point, and inner
solver, and use the model in \eqref{eq:app-model}. The Jacobian surrogates
and model levels are
\begin{equation}
  \underbrace{\nabla F(v_k)}_{\text{Exact-J},\;\delta=0},\;\;
  \underbrace{\nabla F_1(v_k)}_{\text{Local-J},\;\delta=\beta},\;\;
  \underbrace{\nabla F_1(v_k)+B_k}_{\text{uncertified},\;\bar\delta=2\beta},\;\;
  \underbrace{J_{k,r}}_{\text{certified variants}} ,
  \label{eq:surrogates}
\end{equation}
where $B_k=\Pi^{\mathrm F}_{\mathcal B_\beta}(Y_kS_k^\dagger)$ is the projected
residual-QN approximation \eqref{eq:projected-history-prior}. Equation
\eqref{eq:Jkr} completes the actions learned at the current point to the
monotone surrogate $J_{k,r}$. The uncertified surrogate is
\textit{Residual multisecant}. The labels \textit{Certified} and
\textit{Certified $+$ QN} denote the zero and recycled choices of $B_k$.
Each is run with $\bar\delta=2\beta$ and with the step-dependent model level
$\delta^{\rm dir}_{k,r}$, once per trial. The tables write $\delta^{\rm dir}_k$
for the level of the accepted trial. We also include Exact-J with a lazy
refresh every fifth iteration, and Extragradient on $F$ with $\alpha=1/L_0$.
A stale Jacobian is an empirical baseline. The two certified model levels
share the directional outer parameter $\rho_k=(L_1/2)\norm{s_k}$.
The fixed positive level $\bar\delta$ is an empirical control for the inner
regularization. Theorem~\ref{thm:main} uses the bound on
$\delta^{\rm dir}_{k,r}$ at the accepted step and a valid $L_1$.

The measured accuracy is the Stampacchia gap
$G_{\rm S}(z)=\sup_{u\in\Z}\ip{F(z)}{z-u}$, evaluated through the support
function of $\Z$. Each target comparison uses
$G_{\rm S}(\widetilde z_T)\le10^{-3}G_{\rm S}(z_0)$.
The common starting point gives the same threshold for all partitions of
one data set. Monotonicity yields $\Gap(z)\le G_{\rm S}(z)$ in absolute
units, and \eqref{eq:gap} uses the Minty gap.

\begin{figure}[t]
\centering
\includegraphics[width=0.98\textwidth]{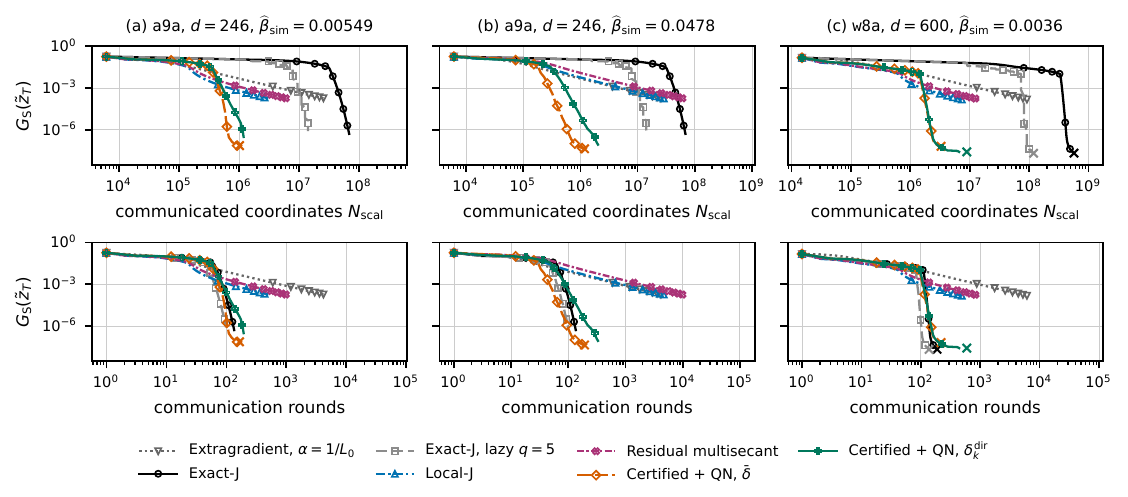}
\caption{Stampacchia gap of the same runs against transmitted scalars
(top) and communication rounds (bottom) on (a) \texttt{a9a} at $h=0$,
(b) \texttt{a9a} at $h=1$, and (c) \texttt{w8a} at $h=0$. We use seed 0
for the three representative partitions. \NumericalFigRefusal}
\label{fig:convergence}
\end{figure}

\begin{table}[t]
\caption{Communication to reach $G_{\rm S}(\widetilde z_T)\le10^{-3}\,G_{\rm S}(z_0)$.
Means over five seeds in each \texttt{a9a} setting and one partition
on \texttt{w8a}. \NumericalTableCounts}
\label{tab:tolerance}
\centering
\footnotesize
\setlength{\tabcolsep}{4pt}
\begin{tabular}{l rr rr rr}
\hline
 & \multicolumn{2}{c}{\texttt{a9a}, $h=0$} & \multicolumn{2}{c}{\texttt{a9a}, $h=1$} & \multicolumn{2}{c}{\texttt{w8a}, $h=0$}\\
Method & $T\downarrow$ & $N_{\rm scal}\downarrow$ & $T\downarrow$ & $N_{\rm scal}\downarrow$ & $T\downarrow$ & $N_{\rm scal}\downarrow$\\
\hline
Extragradient, $\alpha=1/L_0$ & 2\,084 & $2.56\!\cdot\!10^{7}$ & 2\,084 & $2.56\!\cdot\!10^{7}$ & 3\,068 & $9.20\!\cdot\!10^{7}$\\
Exact-J & 33.0 & $5.03\!\cdot\!10^{7}$ & 33.0 & $5.03\!\cdot\!10^{7}$ & 42.0 & $3.79\!\cdot\!10^{8}$\\
Exact-J, lazy $q=5$ & 33.0 & $1.10\!\cdot\!10^{7}$ & 33.0 & $1.10\!\cdot\!10^{7}$ & 42.0 & $8.23\!\cdot\!10^{7}$\\
Local-J & 245 & $3.01\!\cdot\!10^{6}$ & 2\,295 & $2.82\!\cdot\!10^{7}$ & 247 & $7.41\!\cdot\!10^{6}$\\
Residual multisecant & 483 & $5.95\!\cdot\!10^{6}$ & 4\,875 & $6.00\!\cdot\!10^{7}$ & 482 & $1.45\!\cdot\!10^{7}$\\
Certified, $\bar\delta$ & 25.8 & $5.54\!\cdot\!10^{5}$ & 15.0 & $3.63\!\cdot\!10^{5}$ & 37.0 & $1.81\!\cdot\!10^{6}$\\
Certified $+$ QN, $\bar\delta$ & 25.8 & $5.50\!\cdot\!10^{5}$ & 15.0 & $3.62\!\cdot\!10^{5}$ & 37.0 & $1.81\!\cdot\!10^{6}$\\
Certified, $\delta_k^{\rm dir}$ & 31.4 & $6.73\!\cdot\!10^{5}$ & 28.6 & $7.17\!\cdot\!10^{5}$ & 37.0 & $1.80\!\cdot\!10^{6}$\\
Certified $+$ QN, $\delta_k^{\rm dir}$ & 31.4 & $6.73\!\cdot\!10^{5}$ & 28.4 & $6.96\!\cdot\!10^{5}$ & 37.0 & $1.80\!\cdot\!10^{6}$\\
\hline
\end{tabular}

\end{table}

\paragraph{Communication, and the two cost axes.}
\label{sec:comm}
\NumericalComparison\space
In Table~\ref{tab:tolerance}, Residual multisecant uses about 1.95--2.12 times
as many scalar coordinates as Local-J. The two methods use different model
levels, $2\beta$ and $\beta$, respectively.
A full-Jacobian iteration transmits $M(d^2+2d)$ nominal coordinates,
whereas one Extragradient iteration transmits $2Md$. Their ratio is
$(d+2)/2$, equal to 124 on \texttt{a9a} and 301 on \texttt{w8a}.
The communication-round and payload comparisons use the two rows of
Figure~\ref{fig:convergence} and the same accepted trajectories. The
strength of the prescribed regularization can change the step and
outer weight; an iteration advantage alone does not isolate the effect
of the Jacobian surrogate. In the fixed-level comparison on \texttt{a9a},
Exact-J and Certified use the same $\bar\delta$ and both average $25.8$ steps
at $h=0$ and $15.0$ at $h=1$. In the adaptive comparison they apply the same
rule to their own candidates and choose their model levels separately.
Their means are $31.4$ and $31.4$ at $h=0$, and $28.8$ and $28.6$ at $h=1$,
respectively.

\begin{figure}[t]
\centering
\includegraphics[width=0.98\textwidth]{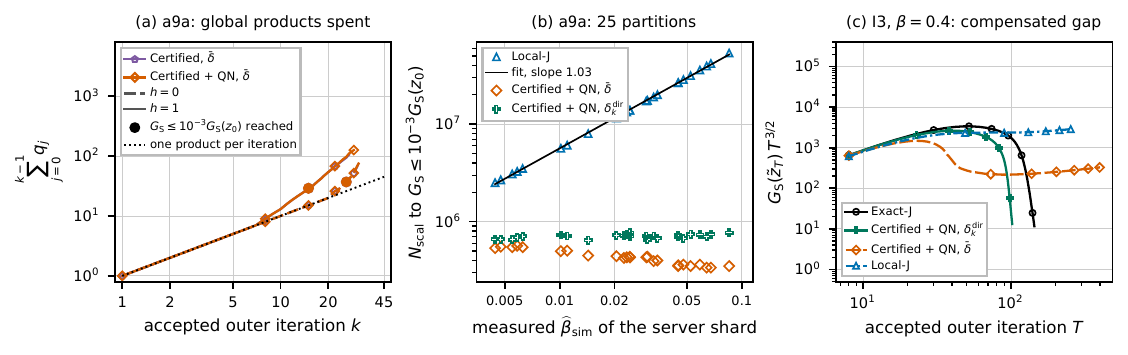}
\caption{(a) Cumulative global JVPs $\sum_{j=0}^{k-1}q_j$ after $k$ accepted
iterations for the fixed model level, with and without the QN prior, on
\texttt{a9a} at $h=0$ (dashed) and $h=1$ (solid). Filled markers show the
first accepted point at the target. (b) Scalar coordinates to the target
against the measured radius \eqref{eq:beta-hat}, on the same 25 partitions
for each method. \NumericalPanelBRefusal\ (c) Compensated Stampacchia gap
on the rate copy of I3 at $\beta=0.4$. The panel describes the measured
finite range on the original axes.}
\label{fig:mechanism}
\end{figure}

\paragraph{The rate, and where the claim ends.}
Figure~\ref{fig:mechanism}(c) reports $T^{3/2}G_{\rm S}(\tilde z_T)$ for
the original rate instance. An increase in this diagnostic over a finite
window is compatible with a slower observed decay on that window.
The rate and assumptions of Theorem~\ref{thm:main} follow from its proof.
Instance~I1 probes the difficult regime where a poorly aligned local
Jacobian can require many successive products and repeated inner solves.
\NumericalStress

\section{Limitations}
\label{sec:discussion}

Several limitations remain. \emph{(i)} The rate requires a valid upper bound
$L_1$ because the proof uses it in the Taylor remainder. The radius $\beta$
controls the residual prior and the guarantee for the uniform baselines.
\emph{(ii)} The worst-case bound $\norm{H_k}_{\op}\le2\beta$ alone does not
order the baseline and the residual quasi-Newton approximation.
\emph{(iii)} We assume full participation and exact products; partial
participation would need a stochastic analysis.  \emph{(iv)} The certificate
products are sequential, so the volume can be far below that of a full Jacobian
while the number of rounds is higher.  \emph{(v)} As is standard, one
approximately solved subproblem counts as one iteration; on a constrained set
that subproblem carries the projection and may dominate the running time.
\emph{(vi)} The bound of Theorem~\ref{thm:instance} depends on the failed
trial directions and lengths. \emph{(vii)} The effect of reusing secants
and products depends on the problem and the history policy. We report
accuracy and communication together.
Finally, the experiments cover two LIBSVM and two synthetic instances; a
broader claim would need more task families and measured wall-clock time.

\section{Conclusion}
\label{sec:conclusion}

We placed a second-order method for distributed variational inequalities
between two standard options. The coefficient of the $T^{-1}$ term in the
local-Jacobian bound depends on $\beta$, and the residual model uses secants
already communicated. With the step-dependent level of Theorem~\ref{thm:main}
and a valid upper bound $L_1$, the certified version attains the classical
$O(L_1D^3T^{-3/2})$ rate while transmitting vectors of length $d$. At a common
horizon, it uses fewer nominal scalar coordinates than repeated full
Jacobian collection whenever the mean number of products is below $d$.

One idea does this work. Appendix~\ref{app:lower} gives a worst-case bound
for deterministic uniform matrix recovery. A directional certificate tests
the error on the current candidate and supplies a new exact Jacobian action
when the test fails. Two directions remain open:
removing the known constants and allowing partial participation, and proving a
lower bound on the curvature information a distributed method must
communicate as a function of $\beta$, $d$, and the network, against which the
count of Theorem~\ref{thm:communication} could be judged.

\section*{Acknowledgments}

The research was supported by Russian Science Foundation (project
No.~23-11-00229), \url{https://rscf.ru/en/project/23-11-00229/}.




\clearpage
\appendix

\section{Mapping to the VIJI inner model}
\label{app:mapping}

At a trial point $v$, the uniform-error baseline uses a surrogate $J$
satisfying
\begin{equation}
  \norm{\nabla F(v)-J}_{\op}\le\delta.
  \label{eq:app-global-envelope}
\end{equation}
The VIJI local model of Agafonov et al.~\cite{agafonov2024} is
\begin{equation}
  \Omega_{v,J,\delta}(x)
  :=F(v)+J(x-v)+\bigl(10\delta+5L_1\norm{x-v}\bigr)(x-v),
  \label{eq:app-model}
\end{equation}
and its stopping condition is
\begin{equation}
  \sup_{z\in\Z}\ip{\Omega_{v,J,\delta}(x)}{x-z}
  \le\frac{L_1}{2}\norm{x-v}^3+\delta\norm{x-v}^2 .
  \label{eq:app-inner}
\end{equation}
Both use the same model level. Here $\delta$ is the regularization level
in the model. We write $\rho_k$ for the separate outer parameter called
$\beta_k$ in the source, to distinguish it from the similarity radius.
For the uniform theorem $\rho_k=\delta$. For a nonzero step use
$\lambda_k(L_1\norm{x_k-v_k}/2+\rho_k)\in[1/32,1/22]$. If the
inner gap is zero at $x=v$, that anchor is a solution.
If $L_1=\delta=0$, the restricted model equals the original affine operator,
and its solution can be returned directly.

Algorithm~\ref{alg:certified-refinement} uses the same form of the model
and inner test on $E$, with $F(v_k)$ replaced by $\Pi_EF(v_k)$ and
$J=J_{k,r}$. The global bound \eqref{eq:app-global-envelope} is unnecessary
for this monotone completion. The accepted candidate passes three tests:
\begin{align*}
 &\sup_{z\in\Z}\ip{\Omega_{v_k,J_{k,r},\delta^{\rm dir}_{k,r}}(x_{k,r})}{x_{k,r}-z}
   \le\tfrac{L_1}{2}\norm{s_{k,r}}^3+\delta^{\rm dir}_{k,r}\norm{s_{k,r}}^2 ,\\
 &0\le\delta^{\rm dir}_{k,r}\le\tfrac{L_1}{2}\norm{s_{k,r}},\\
 &\norm{(A_k-J_{k,r})s_{k,r}}\le\tfrac{L_1}{2}\norm{s_{k,r}}^2 .
\end{align*}
The level and the inner gap are evaluated on the same candidate before
the JVP. A failed directional test supplies the next correction, after
which both local tests are evaluated again. The fixed level
$\bar\delta=2\beta$ alone does not imply the second test.

\section{Proofs of the main results}
\label{app:proofs}

\subsection{Projected residual quasi-Newton approximation}

\proof[of Proposition~\ref{prop:qn-envelope}]
By \eqref{eq:beta}, $\norm{R_k}_{\op}\le\beta$, hence $R_k\in\mathcal B_\beta$. By definition of metric projection, $B_k\in\mathcal B_\beta$ as well. Therefore
\[
  \norm{R_k-B_k}_{\op}
  \le\norm{R_k}_{\op}+\norm{B_k}_{\op}
  \le2\beta.
\]
Since $\mathcal B_\beta$ is closed and convex and contains $R_k$, the metric projection of $\widehat B_k$ onto $\mathcal B_\beta$ cannot increase its Frobenius distance to $R_k$:
\[
  \norm{R_k-B_k}_{\mathrm F}
  \le
  \norm{R_k-\widehat B_k}_{\mathrm F}.
\]
This proves \eqref{eq:projected-residual-error} and \eqref{eq:projection-nonexpansive}. \qed

\subsection{Proofs for stored residual secants and JVPs}

\proof[of Theorem~\ref{thm:recycled-secants}]
Let $P=P_S=SS^\dagger$ and $\widehat B=YS^\dagger$.  Since $P$ is the orthogonal projector onto $\operatorname{range}(S)$,
\[
  R_k-\widehat B
  =R_k(I-P)+R_kP-YS^\dagger
  =R_k(I-P)+(R_kS-Y)S^\dagger,
\]
which proves \eqref{eq:history-decomposition}.  The first summand has right support in $\operatorname{range}(S)^\perp$, whereas the second satisfies
\[
  (R_kS-Y)S^\dagger=(R_kS-Y)S^\dagger P.
\]
Hence their Frobenius inner product is zero, and \eqref{eq:history-pythagoras} follows.

For the $j$th secant, the fundamental theorem of calculus gives
\[
  y_j
  =\int_0^1\nabla G(a_j+t s_j)s_j\,dt.
\]
Thus
\[
  R_ks_j-y_j
  =\int_0^1
  \bigl[R_k-\nabla G(a_j+t s_j)\bigr]s_j\,dt.
\]
Using \eqref{eq:residual-jac-lip},
\[
\begin{split}
  \norm{R_ks_j-y_j}
  &\le
  L_G\int_0^1
  \norm{v_k-a_j-t s_j}\,\norm{s_j}\,dt\\
  &\le
  L_G\left(\norm{v_k-a_j}+\frac12\norm{s_j}\right)\norm{s_j}.
\end{split}
\]
Summing squared column bounds yields
\[
  \norm{R_kS-Y}_{\mathrm F}^2
  \le
  L_G^2\sum_{j=1}^{m}
  \left(\norm{v_k-a_j}+\frac12\norm{s_j}\right)^2\norm{s_j}^2.
\]
Finally,
\[
  \norm{(R_kS-Y)S^\dagger}_{\mathrm F}
  \le
  \norm{R_kS-Y}_{\mathrm F}\norm{S^\dagger}_{\op}.
\]
Projection onto the closed convex set $\mathcal B_\beta$ cannot increase the Frobenius distance to $R_k\in\mathcal B_\beta$. Together with \eqref{eq:projection-nonexpansive}, the preceding estimate gives \eqref{eq:history-bound}. \qed

\proof[of Corollary~\ref{cor:recycled-jvp}]
Let $P_U=UU^{\mathsf T}$ and first consider $\widehat B_k^{\rm jvp}$.  Since the columns of $U$ are orthonormal,
\[
  \widehat B_k^{\rm jvp}P_U=\widehat B_k^{\rm jvp},
  \qquad
  (R_k-\widehat B_k^{\rm jvp})(I-P_U)=R_k(I-P_U).
\]
On the current correction subspace,
\[
  (R_k-\widehat B_k^{\rm jvp})U
  =\bigl[(R_k-R(v_1))u_1,\ldots,(R_k-R(v_m))u_m\bigr].
\]
The two domain components are orthogonal, and therefore
\[
\begin{split}
  \norm{R_k-\widehat B_k^{\rm jvp}}_{\mathrm F}^2
  &=\norm{R_k(I-P_U)}_{\mathrm F}^2
  +\sum_{j=1}^{m}\norm{(R_k-R(v_j))u_j}^2\\
  &\le\norm{R_k(I-P_U)}_{\mathrm F}^2
  +L_G^2\sum_{j=1}^{m}\norm{v_k-v_j}^2.
\end{split}
\]
Projection onto $\mathcal B_\beta$ cannot increase the Frobenius distance because $R_k\in\mathcal B_\beta$. This proves \eqref{eq:recycled-jvp-bound}. \qed

\subsection{Proof of the uncertified local-Jacobian + QN rate}

\proof[of Proposition~\ref{prop:combined-rate}]
By \eqref{eq:combined-qn} and \eqref{eq:Hk},
\[
  \nabla F(v_k)-J_k^{\rm QN}=R_k-B_k=H_k.
\]
Therefore the hypothesis $\norm{H_k}_{\op}\le\delta_{\rm QN}$ is exactly the uniform inexact-Jacobian hypothesis in Theorem~\ref{thm:known-uniform}.  Using the corresponding uniform-theorem outer parameter $\rho_k=\delta_{\rm QN}$ and substituting $\delta=\delta_{\rm QN}$ into \eqref{eq:known-rate} gives \eqref{eq:combined-rate}.  Proposition~\ref{prop:qn-envelope} allows one to take $\delta_{\rm QN}=2\beta$ for the projected residual-QN specialization. \qed

\paragraph{Uniform-error estimate used in Theorem~\ref{thm:known-uniform}.}
Let $a_k=\norm{s_k}$, $h_k=\norm{v_{k+1}-v_k}$, and
$d_k=L_1a_k/2+\delta$. For $d_k>0$, the uniform error bound and the
Taylor formula give
\begin{equation*}
  \norm{F(x_k)-F(v_k)-J_ks_k}\le d_k a_k.
\end{equation*}
Evaluate the
inner criterion at $v_{k+1}$ and use
$x_k-v_{k+1}=s_k-(v_{k+1}-v_k)$. The criterion contributes at most
$d_ka_k^2$, the Taylor residual contributes at most
$d_ka_k(a_k+h_k)$, and the inner regularizer contributes at most
$-10d_ka_k^2+10d_ka_kh_k$. Hence
\[
  \ip{F(x_k)}{x_k-v_{k+1}}\le d_k(-8a_k^2+11a_kh_k).
\]
Completing the square after multiplying by $\lambda_k$ and subtracting
$h_k^2/2$ gives a bound by
$\bigl(-8c_k+(121/2)c_k^2\bigr)a_k^2\le-a_k^2/8$ for
$c_k=\lambda_kd_k\in[1/32,1/22]$. The same potential argument as in the
proof of Theorem~\ref{thm:main} below yields
\[
  \sum_{k<T}\lambda_k\ip{F(x_k)}{x_k-y}+\tfrac18\sum_{k<T}a_k^2
  \le\tfrac12\norm{y-z_0}^2\qquad(y\in\Z).
\]
Taking $y=z^\star$ gives $\sum_{k<T}a_k^2\le4D^2$. By
Cauchy--Schwarz, $\sum_{k<T}a_k\le2D\sqrt T$ and
$\sum_{k<T}\lambda_k\ge T^2/[32(L_1D\sqrt T+\delta T)]$.
The Minty-gap bound follows by dividing the pairing estimate by
$\sum_{k<T}\lambda_k$, after replacing $F(x_k)$ by $F(y)$ via
monotonicity. It gives $16L_1D^3/T^{3/2}+16\delta D^2/T$, which implies
\eqref{eq:known-rate}. If $d_k=0$ or $x_k=v_k$, the direct-return cases
in Appendix~\ref{app:mapping} apply. The directional estimate for the
weight used in this paper is shown in full next. \qed

\subsection{Refinement invariant and finite termination}

\proof[of Theorem~\ref{thm:main}]
Fix $k$ and drop it from the trial indices. If $z$ belongs to the relative
interior of $\Z$ and $u\in E$, monotonicity applied to $z$ and $z+tu$ for
small $t>0$ gives $u^{\mathsf T}\nabla F(z)u\ge0$ in the limit. Continuity
of the derivative gives the same inequality at the anchor, even when it is
on the boundary. Thus $\operatorname{sym}(A_k)\succeq0$ on $E$.
Assume $U_r$ has orthonormal columns and $W_r=\nabla F(v_k)U_r$, as holds
with empty columns. Put $A=A_k$ and $H=H_k$. Then
\begin{equation}
  P_r=U_rU_r^{\mathsf T},\qquad Q_r=I_E-P_r,\qquad C_r=HP_r.
  \label{eq:app-projector}
\end{equation}
The known terms in \eqref{eq:Jkr} equal $AP_r-P_rA^{\mathsf T}Q_r$.
The sum of its off-diagonal terms $Q_rAP_r-P_rA^{\mathsf T}Q_r$ is
skew-symmetric. Consequently $J_{k,r}U_r=AU_r$ and
\begin{equation}
  \operatorname{sym}(J_{k,r})=P_r\operatorname{sym}(A)P_r
            +Q_r\operatorname{sym}(J_k^+)Q_r\succeq0.
  \label{eq:app-invariant}
\end{equation}
This verifies \eqref{eq:refinement-invariant} at trial $r$.

For any candidate step $s_r$, certification is exact when $Q_rs_r=0$.
Otherwise write $w_r=Q_rs_r$ and $u_{r+1}=w_r/\norm{w_r}$. A failed test
uses the returned $g=\nabla F(v_k)s_r$ to obtain
\begin{equation}
  \frac{g-W_rU_r^{\mathsf T}s_r}{\norm{w_r}}
  =\frac{\nabla F(v_k)(s_r-P_rs_r)}{\norm{w_r}}
  =\nabla F(v_k)u_{r+1}.
  \label{eq:app-new-column}
\end{equation}
The new direction has unit norm and is orthogonal to $U_r$, and the new
column of $W_r$ is exact. Subtracting the prior action from that column
updates $C_{r+1}=HP_{r+1}$. This proves the induction. At most $m_E$ tests
can fail. If $m_E$ tests fail, $Q_{m_E}=0$ and the next candidate needs no
JVP. If fewer fail, at most one additional query is used for acceptance.
The local solves reach an admissible candidate as described in
Section~\ref{sec:method}, using level zero if needed. This proves the JVP
bound including the final acceptance test. An accepted $s_r=0$ would
satisfy the original VI, since then the model equals $\Pi_EF(v_k)$ and the
inner gap is zero. This case is detected by the anchor check.

For the convergence estimate, let $a_k=\norm{s_k}>0$ and let $\delta_k$
be the accepted trial's model level. By the integral Taylor formula and
\eqref{eq:directional-condition}, the linearization residual satisfies
\[
 r_k:=\Pi_E(F(x_k)-F(v_k))-J_{k,r}s_k,\qquad
 \norm{r_k}\le\frac{L_1}{2}a_k^2+\frac{L_1}{2}a_k^2=L_1a_k^2 .
\]
Let $t_k=v_{k+1}-v_k$ and $h_k=\norm{t_k}$. Evaluating
\eqref{eq:app-inner} at $v_{k+1}$, using
$x_k-v_{k+1}=s_k-t_k$, gives
\begin{align*}
  \ip{F(x_k)}{x_k-v_{k+1}}
  &\le \tfrac{L_1}{2}a_k^3+\delta_ka_k^2
    +\ip{r_k-(10\delta_k+5L_1a_k)s_k}{s_k-t_k}\\
  &\le-\tfrac72 L_1a_k^3-9\delta_ka_k^2
        +(6L_1a_k^2+10\delta_ka_k)h_k .
\end{align*}
Here $\ip{r_k}{s_k-t_k}\le L_1a_k^2(a_k+h_k)$ and
$\ip{s_k}{t_k}\le a_kh_k$. Write $\theta_k=\delta_k/(L_1a_k)$, so
$0\le\theta_k\le1/2$. Since $\lambda_k=1/(27L_1a_k)$, we obtain by
completing the square
\begin{align*}
 \lambda_k\ip{F(x_k)}{x_k-v_{k+1}}-\tfrac12h_k^2
 &\le -\left(\frac{17}{162}+\frac{\theta_k(366-100\theta_k)}{1458}\right)a_k^2
       -\tfrac12\left(h_k-\frac{6+10\theta_k}{27}a_k\right)^2\\
 &\le -\frac{17}{162}a_k^2.
\end{align*}

Let $\Phi(b)=\max_{y\in\Z}\{\ip{b}{y-z_0}-\norm{y-z_0}^2/2\}$.
Its unique maximizer at $b_k$ is $v_k$. The objective is
$1$-strongly concave in $y$, so substituting $v_{k+1}$ into the maximization
at $b_k$ loses at least $h_k^2/2$. The outer update therefore gives
\[
 \Phi(b_{k+1})-\Phi(b_k)
 \le -\lambda_k\ip{F(x_k)}{v_{k+1}-z_0}-\tfrac12 h_k^2 .
\]
Adding $\lambda_k\ip{F(x_k)}{x_k-z_0}$ and using the preceding descent
bound yields
$\Phi(b_{k+1})-\Phi(b_k)+\lambda_k\ip{F(x_k)}{x_k-z_0}
 \le-(17/162)a_k^2$.
Sum over $k<T$, use $\Phi(0)=0$, and test the final maximization at an
arbitrary $y\in\Z$. Because $b_T=-\sum_{k<T}\lambda_kF(x_k)$, this gives
\[
 \sum_{k<T}\lambda_k\ip{F(x_k)}{x_k-y}
 +\frac{17}{162}\sum_{k<T}a_k^2
 \le \frac12\norm{y-z_0}^2.
\]
Continuity on the compact convex set gives a VI solution $z^\star$,
for example by Brouwer's theorem applied to the projected operator step.
With $y=z^\star$, monotonicity and the VI give each weighted pairing
nonnegative. Thus $\sum_{k<T} a_k^2\le(81/17)D^2$. Cauchy--Schwarz
first gives $\sum a_k\le \sqrt{T\sum a_k^2}$ and then
$T^2\le(\sum a_k)(\sum a_k^{-1})$. Accordingly,
\[
 \sum_{k<T}\lambda_k\ge
 \frac{T^{3/2}}{27L_1(\sum_{k<T}a_k^2)^{1/2}}.
\]
For every $y\in\Z$, monotonicity also gives
$\ip{F(y)}{x_k-y}\le \ip{F(x_k)}{x_k-y}$. Applying this to the summed
inequality, dividing by $\sum_{k<T}\lambda_k$ and taking the supremum
over $y$ bounds the Minty gap at $\widetilde z_T$ by
$243 L_1D^3/(2\sqrt{17}\,T^{3/2})\le32L_1D^3/T^{3/2}$. \qed

\subsection{Proof of Theorem~\ref{thm:instance}}

\proof
At fixed $k$, put $A=A_k$, $J^+=J_k^+$ and suppress $k$ in the trial indices.
Expanding \eqref{eq:Jkr} with $P_r+Q_r=I_E$ gives
\[
 A-J_{k,r}=2P_r\operatorname{sym}(A)Q_r+Q_r(A-J^+)Q_r .
\]
On a failed trial $\norm{e_{k,r}}>L_1\norm{s_{k,r}}^2/2$ and
$Q_rs_{k,r}=\norm{w_{k,r}}u_{k,r+1}$. Since
$\norm{w_{k,r}}\le\norm{s_{k,r}}$, and the two terms in the displayed
identity have orthogonal ranges, one has
\[
 \tfrac{L_1^2}{4}\norm{s_{k,r}}^2
 <4\norm{P_r\operatorname{sym}(A)u_{k,r+1}}^2
       +\norm{Q_r(A-J^+)u_{k,r+1}}^2 .
\]
The directions $u_{k,r+1}$ from the failures are orthonormal. The sum of
unlearned contributions is at most $\norm{A-J^+}_{\mathrm F}^2$.
Projection of the symmetric part of $J_k^0$ onto the positive semidefinite
cone cannot increase its Frobenius distance to $\operatorname{sym}(A)$,
a point in that cone. As the skew part is unchanged,
\[
 \norm{A-J^+}_{\mathrm F}\le\norm{A-J_k^0}_{\mathrm F}
   =\norm{\Pi_EH_k|_E}_{\mathrm F}\le\norm{H_k}_{\mathrm F}.
\]
Summing the learned terms yields $\mathcal C_k$. They count each strictly
upper triangular entry of the symmetric compression to the learned
orthonormal basis exactly once, with coefficient $4$. Each entry also has
a matching transposed entry, so
$\mathcal C_k\le2\norm{\operatorname{sym}(A)}_{\mathrm F}^2$.
This proves \eqref{eq:energy-budget}, also when the set of failures is empty.
If it is nonempty, $|\mathcal F_k|\tau_k^2\le
\sum_{r\in\mathcal F_k}\norm{s_{k,r}}^2$. At most one successful query
follows the failures, and the total is capped by $d$ by
Theorem~\ref{thm:main}. These bounds give \eqref{eq:q-bound}. \qed

\proof[of Corollary~\ref{cor:history-jvp}]
Substitute the upper bound \eqref{eq:history-bound} for the prior error
$\norm{H_k}_{\mathrm F}^2=\norm{R_k-B_k}_{\mathrm F}^2$ in
\eqref{eq:q-bound}. The bound retains $\mathcal C_k$ and uses the same
failed-step scale $\tau_k$. \qed

\subsection{Proof of Theorem~\ref{thm:communication}}

\proof
By Theorem~\ref{thm:main}, $T_\varepsilon$ accepted steps suffice unless the
method has already returned an exact solution. Each accepted step uses the
two reductions $F(v_k)$ and $F(x_k)$, plus $q_k$ sequential JVPs. If the
anchor check terminates after $T<T_\varepsilon$ accepted steps, the last
call uses at most one additional $F(v_T)$ reduction, and
$2T+1\le 2T_\varepsilon$. This proves
\eqref{eq:cert-comm}--\eqref{eq:cert-oracles}. Summing $q_k\le d$ gives
$Q_T\le dT$. Summing \eqref{eq:q-bound}, with $q_k\le1$ for
$\mathcal F_k=\varnothing$, gives the other term in \eqref{eq:pathwise-Q}.
Use \eqref{eq:history-bound} in that sum when the stated secant
assumptions hold.

The exact-Jacobian variant with zero model level has vanishing directional
error, so the same descent proof applies with the same outer weight.
It uses two vector reductions and one matrix reduction per accepted step,
which gives \eqref{eq:exact-comm} at the stated horizon. Substitution of
$Q_{T_\varepsilon}\le dT_\varepsilon$ proves
\eqref{eq:payload-domination}. At any common horizon $T$, the ratio equals
$Md(2T+Q_T)/(M(2d+d^2)T)$, as in \eqref{eq:bandwidth-factor}. Its numerator
is strictly smaller when $Q_T/T<d$. \qed

\proof[of Corollary~\ref{cor:bandwidth-factor}]
For every common accepted-step horizon $T\ge1$, sum the certified and exact
costs per step and divide using the preceding formulas. The ratio is
$(2+Q_T/T)/(d+2)$, which tends to zero with $d\to\infty$ when
$Q_T/T=o(d)$. \qed

\section{A worst-case obstruction for uniform matrix reconstruction}
\label{app:lower}

The certified method requires accuracy only in the accepted step direction rather than a uniformly accurate Jacobian matrix.  The following elementary lower bound explains why this distinction is necessary without additional spectral assumptions.

\begin{teo}[Deterministic JVP obstruction]
\label{thm:lower}
Let an unknown matrix $A\in\R^{d\times d}$ satisfy $\norm{A}_{\op}\le\Delta$.  Any deterministic adaptive procedure that makes $q<d$ JVP queries and returns $\widehat A$ has a consistent instance for which
\[
  \norm{A-\widehat A}_{\op}\ge\Delta.
\]
\end{teo}

\proof
Run the procedure on the all-zero reply transcript.  The resulting query vectors span a proper subspace, so choose a unit vector $w$ orthogonal to it and any unit vector $a$.  Both matrices
\[
  A_+=\Delta aw^{\mathsf T},
  \qquad
  A_-=-\Delta aw^{\mathsf T}
\]
produce the same transcript.  Their distance is $2\Delta$, so at least one is at distance at least $\Delta$ from the common output $\widehat A$. \qed

For $M\ge2$ the same construction is realized by strongly monotone linear
operators. Write $\omega=\sum_{i>1}w_i>0$, set $F_1(z)=\mu z$, and let
$F_i(z)=\mu z+A_\pm z/\omega$ for every $i>1$, with $\mu>\Delta/\omega$.
The global operator is $F(z)=\mu z+A_\pm z$. Its server Jacobian is the
same for both instances, and its residual Jacobian is $A_\pm$. Any sequence
of queried residual actions in the proof is therefore identical on both
instances. This proves an obstruction for uniform matrix reconstruction.
The directional test of Theorem~\ref{thm:main} requires only the action
in an accepted step direction.

\section{Additional numerical certificate verification}
\label{app:reproducibility}

\paragraph{The main instance in full.}
Node~$i$ contributes the operator of the average over its shard with the weight
$w_i=N_i/N$ of \eqref{eq:finite-sum}. Within each label the class shards
are balanced up to rounding, and the features are normalized to $\norm{a_j}\le1$ before the split, so the displayed problem is
the same for every partition and only its decomposition changes.  The symmetric
part of the Jacobian is block diagonal, its $r$ block being
$\gamma I-\bar c\,ww^{\mathsf T}$ with $\bar c=\overline{\ell''}\le\tfrac14$, so
$\gamma=R_w^2/4$ is the smallest value for which the operator is monotone, and
that is the value taken.  With $\lambda=0$ the $w$ block is
$\tfrac1n\sum_jc_j(a_j+r)(a_j+r)^{\mathsf T}$ with $c_j>0$, whose rank can increase by at most one from
that of the design matrix. The numerical ranks at the origin are
108 of $123$ on \texttt{a9a} and 266 of $300$ on \texttt{w8a}. We use
monotonicity of these instances without a strong-monotonicity estimate. The
heterogeneity family below varies the feature distribution across shards. The
positive-label proportion in each shard is approximately 0.2404--0.2412 on
\texttt{a9a} and 0.0296--0.0301 on \texttt{w8a}. We use
$h\in\{0,\tfrac14,\tfrac12,\tfrac34,1\}$ with five seeds on \texttt{a9a},
25 partitions in all; \texttt{w8a} enters the study for its dimension and
not for its heterogeneity and is run at $h=0$.

\paragraph{The auxiliary instances.}
Two further instances isolate the extreme regimes.  Instance~I1 is the
cubic-regularized bilinear saddle point
$\min_x\max_y\,y^{\mathsf T}(Ax-b)+\tfrac{\rho}{6}\norm{x}^3$ with
$\rho=10^{-3}$, $d=100$ and $A$ upper bidiagonal, whose nonzeros are split over
$M=8$ workers, so that the server holds a $1/M$ fraction of the coupling and
$\beta=9.95$, about five times $\norm{A}_{\op}=2.00$: the worst case
for a residual surrogate, since the local Jacobian is then further from the global
Jacobian than the zero matrix is.  Instance~I3, used wherever
$\beta$ or $d$ has to be varied, is the synthetic benchmark with controlled similarity
$F(z)=Az+\rho z^{\odot3}$ on $[-1,1]^d$ with $F_1=F-G$, $F_i=F+G/(M-1)$ and
$G(z)=aE\tanh(z/a)$, $\norm{E}_{\op}=\beta$; the average is exactly $F$, so
$\beta$ changes the server Jacobian and nothing else. The original
similarity sweep varies $\beta\in\{0.02, 0.08, 0.2, 0.4, 0.8\}$ at
$d=18$ and is recorded separately. The dimension comparison below runs
$d\in\{18, 40, 80, 160, 320\}$ at $\beta=0.2$. The rate study uses a copy of I3 with the same
algebraic form and a different conditioning, stated there, since every constant
of the rate plot depends on it.

\paragraph{The methods in full.}
In \eqref{eq:surrogates}, the approximation $B_k=\Pi^{\mathrm F}_{\mathcal
B_\beta}(Y_kS_k^\dagger)$ is computed from residual secants, with
$S_k^\dagger$ evaluated by a rank-revealing decomposition and the singular values
clipped at $\beta$ as in \eqref{eq:projected-history-prior}. The pairs of $G$
in \eqref{eq:G} use operator values already available from the outer method.
Exact-J lazy refreshes $\nabla F(v_k)$ every fifth iteration. The two
certified methods differ in the prior: \textit{Certified} uses $B_k=0$, and
\textit{Certified+QN} uses the $m$ most recent residual secants and the $m$ most
recent pairs $(u_{j,l},R_ju_{j,l})$ from earlier failed tests. The buffers are
kept separately and each is trimmed to $3m$, with $m=8$ on I1 and I3 and
$m=20$ on the real data. The fit normalizes the columns of $S_k$ and the
corresponding columns of $Y_k$ by the same factor, discards singular values
below $10^{-6}\sigma_{\max}$, and clips the fit at the stated radius.
The inner subproblem is solved by extragradient with warm starts.
\NumericalInnerProtocol\space
The implementation caps the second-order outer weight $\lambda_k$ at $10^6$.
In the main runs this cap was active only after the first target hit, including
in some of the later iterates shown in Figure~\ref{fig:convergence}.

\paragraph{The step-dependent level.}
The initial proposal for $\delta^{\rm dir}_{k,r}$ comes from the same
server-local search on $[0,\bar\delta]$, comparing the model level with
$(L_1/2)\norm{x_\delta-v_k}$. The search uses ten bisection steps in the
main runs, with the rate settings given below. A finite search with
inexact subproblem solutions proposes a level for testing. In the
corrected implementation, the server evaluates both the scale condition
and the inner stopping test on the actual candidate. It uses the
zero-level model if the proposed level or its candidate fails either
condition. It then verifies \eqref{eq:directional-condition} using the
accepted trial direction, or the previously learned exact action when
available. All model solves and the local scale tests use only server
work. \NumericalFailureProtocol

\paragraph{Constants, and what they are estimates of.}
On the LIBSVM instances $L_1$ is the largest ratio
$\norm{\nabla F(z)-\nabla F(z')}_{\op}/\norm{z-z'}$ observed on a fixed probe
set, 0.349 on \texttt{a9a} and 0.281 on \texttt{w8a}; on I1 and I3 it is
the analytic constant of the cubic term.  On the main instance the similarity
radius is measured on the partition,
\begin{equation}
  \widehat\beta_{\rm sim}
  =\max_{z\in\mathcal T}\norm{\nabla F_1(z)-\nabla F(z)}_{\op},
  \label{eq:beta-hat}
\end{equation}
at six checkpoints $\mathcal T$ on \texttt{a9a} and five on
\texttt{w8a} taken by exact singular value decomposition along one high-accuracy
reference trajectory common to all partitions of that data set, so that it is a
trajectory-restricted lower estimate of the $\beta$ of \eqref{eq:beta} and is
comparable across partitions. Both probe values are lower estimates of
global constants on LIBSVM. The main real-data comparison retains these
empirical inputs, and the proof requires a valid upper bound on $L_1$.
Changing the supplied $L_1$ changes the local model and the tested direction,
so the counts can change as well. On I1 and I3 the given Jacobian Lipschitz
constants are analytic. The normalized natural residual shown in selected
diagnostics is $\norm{z-\Pi_{\Z}(z-\eta F(z))}/
\norm{z_0-\Pi_{\Z}(z_0-\eta F(z_0))}$. The common scale is $\eta=1$ on
the LIBSVM instances, $1/L_0$ on I1, and $1/\norm{A}_{\op}$ on I3.
\NumericalConstantProtocol

\paragraph{The heterogeneity family.}
On the main instance the shards are class-balanced feature shards: for each class
the objects are sorted along its leading principal direction and cut into $M$
contiguous shards, assigned at random, and $h\in[0,1]$ interpolates, each client
keeping a random fraction $h$ of its shard while the rest is pooled and
redistributed. Every object is used exactly once. Sizes are balanced up to
rounding within each class, so $h$ controls feature ordering at nearly fixed
class proportions.

\paragraph{Global JVP counts and reuse of previously computed products.}
Figure~\ref{fig:mechanism}(a) compares the two fixed-level methods at $h=0$
and $h=1$. The first candidate can pass after one product, and a failed
product can trigger a new local solve in the enlarged learned span. An
exact action on the learned span can certify a later candidate without a
new product. Local work in the level search and the remaining inner solves
is accounted separately from nominal reduction costs. In this comparison,
the QN prior combines stored residual secants and stored products.
Its effect on the number of products therefore concerns the combined prior.
\NumericalPanelObservation

\paragraph{The rate.}
Theorem~\ref{thm:main} states a $T^{-3/2}$ bound on the Minty gap, and
Corollary~\ref{cor:local} includes an additional $\beta D^2/T$ term in its
upper bound. Figure~\ref{fig:mechanism}(c) plots the larger measured gap,
$G_{\rm S}(\tilde z_T)$, multiplied by $T^{3/2}$. We evaluate that gap from
the support function of $\Z$. The rate copy of I3 uses $d=18$, $\rho=0.5$,
$L_1=3.0$, $D\approx8.49$, and symmetric eigenvalues of $A$ spread over
$[0.05,2.0]$. We use three replicas with seeds 0--2 and an outer limit of
400. The experiments use $\beta=0.1,0.4$, with $0.4$ shown, an initial
inner budget of 2000, a cap of 40000, and 14 level-search steps.
The panel is a geometric mean over replicas at common displayed iterations.
Values past the recorded accepted prefixes, and values below the original
natural-residual cutoff $3\cdot10^{-7}$ on average, are excluded. The panel
is a finite-range diagnostic on this instance. It leaves asymptotic
sharpness of the stated uniform bound open.
The exponents fitted on iterations 128--257 are 1.23 for Local-J and
1.20 for Certified $+$ QN at the fixed level.
\NumericalRateObservation

\paragraph{Bandwidth factor across dimensions.}
Corollary~\ref{cor:bandwidth-factor} is an identity of the accounting model at
a common accepted horizon. Table~\ref{tab:dimension} reports the observed
$Q_T/T$ on I3 at $\beta=0.2$ over a factor 18 in $d$, along with an
end-to-end payload ratio at a common empirical target where that ratio is
available. The end-to-end comparison uses both accepted iteration counts.

\begin{table}[t]
\caption{Instance~I3 at $\beta=0.2$ on the original five-replica grid.
The tolerance is $G_{\rm S}(\tilde z_T)\le 10^{-3}G_{\rm S}(z_0)$,
and the certified variant uses $\delta^{\rm dir}_{k,r}$. The realized ratio
uses paired replicas reaching the tolerance. \NumericalTableCounts}
\label{tab:dimension}
\centering
\small
\begin{tabular*}{\textwidth}{@{\extracolsep{\fill}}r rr r rr}
\hline
$d$ & \multicolumn{2}{c}{$T$ to tolerance $\downarrow$} & $Q_T/T\downarrow$ & \multicolumn{2}{c}{payload ratio $\downarrow$}\\
 & Certified $+$ QN, $\delta_k^{\rm dir}$ & Exact-J & & realized & $(2{+}Q_T/T)/(d{+}2)$\\
\hline
18 & 32.4 & 50.0 & 1.082 & $1.0\!\cdot\!10^{-1}$ & $1.5\!\cdot\!10^{-1}$\\
40 & 36.2 & 55.4 & 1.000 & $4.7\!\cdot\!10^{-2}$ & $7.1\!\cdot\!10^{-2}$\\
80 & 41.0 & 62.2 & 1.000 & $2.4\!\cdot\!10^{-2}$ & $3.7\!\cdot\!10^{-2}$\\
160 & 47.2 & 71.0 & 1.000 & $1.2\!\cdot\!10^{-2}$ & $1.9\!\cdot\!10^{-2}$\\
320 & 53.4 & 79.2 & 1.000 & $6.3\!\cdot\!10^{-3}$ & $9.3\!\cdot\!10^{-3}$\\
\hline
\end{tabular*}

\end{table}

\label{LastBibItem:}

\label{article_end}
\end{document}